\documentclass [a4paper,10pt]{article}
\usepackage{graphicx}
\usepackage{mathrsfs}
\usepackage[margin=0.86in]{geometry}

\usepackage[utf8]{inputenc}
\usepackage{textcomp}
\usepackage{amsthm}
\usepackage{dsfont}
\usepackage{latexsym}
\usepackage{amssymb}
\usepackage{amsmath}
\usepackage{mathtools}
\usepackage{xfrac}

\usepackage{etoolbox}
\apptocmd{\lim}{\limits}{}{}
\apptocmd{\sup}{\limits}{}{}
\apptocmd{\inf}{\limits}{}{}
\apptocmd{\liminf}{\limits}{}{}
\apptocmd{\limsup}{\limits}{}{}
\pretocmd{\langle}{\left}{}{}
\pretocmd{\rangle}{\right}{}{}

\newcommand{\vect}[1]{\mathbf{#1}}

\newcommand{\Ac}{\mathcal{A}}

\newcommand{\Dc}{\mathcal{D}}
\newcommand{\Ec}{\mathcal{E}}
\newcommand{\Fc}{\mathcal{F}}
\newcommand{\Hc}{\mathcal{H}}
\newcommand{\Kc}{\mathcal{K}}
\newcommand{\Lc}{\mathcal{L}}
\newcommand{\Mc}{\mathcal{M}}
\newcommand{\Sc}{\mathcal{S}}
\newcommand{\Tc}{\mathcal{T}}

\newcommand{\mi}{\mu}
\renewcommand{\phi}{\varphi}

\newcommand{\N}{\mathds{N}}

\newcommand{\R}{\mathds{R}}
\newcommand{\C}{\mathds{C}}
\newcommand{\Hd}{\mathds{H}}

\newcommand{\gf}{\mathfrak{g}}

\newcommand{\zf}{\mathfrak{z}}

\newcommand{\Pf}{\mathfrak{P}}
\newcommand{\Sf}{\mathfrak{S}}
\newcommand{\Uf}{\mathfrak{U}}

\newcommand{\dd}{\mathrm{d}}
\newcommand{\loc}{\mathrm{loc}}

\DeclareMathOperator{\card}{Card}

\DeclarePairedDelimiter{\abs}{\lvert}{\rvert}
\DeclarePairedDelimiter{\norm}{\lVert}{\rVert}
\DeclarePairedDelimiterX\Set[1]\lbrace\rbrace{#1}

\newcommand{\Meg}{\geqslant}
\newcommand{\meg}{\leqslant}
\newcommand{\Supp}[1]{\mathrm{Supp}\left( #1\right)}

\newcommand{\quot}[2]{\mathchoice
	{\setbox1\hbox{${\displaystyle #1}_{\scriptstyle #2}$}
		\text{$#1$\raisebox{-1mm}{\resizebox{2.8mm}{3.5mm}{$/$}}\raisebox{-1.8mm}{$#2$}}}
	{\setbox1\hbox{${\textstyle #1}_{\scriptstyle #2}$}
		\text{$#1$\raisebox{-1mm}{\resizebox{2.8mm}{3.5mm}{$/$}}\raisebox{-1.8mm}{$#2$}}}
	{\setbox1\hbox{${\scriptstyle #1}_{\scriptscriptstyle #2}$}
		\text{$#1$\raisebox{-0.6mm}{\resizebox{2.3mm}{2.0125mm}{$/$}}\raisebox{-1.4mm}{$ #2$}}}
	{\setbox1\hbox{${\scriptscriptstyle #1}_{\scriptscriptstyle #2}$}
		\text{$#1$\raisebox{-0.3mm}{\resizebox{1.55mm}{1.4mm}{$/$}}\raisebox{-0.7mm}{$#2$}}}}

\newcommand{\restr}[2]{\mathchoice
	{\setbox1\hbox{${\displaystyle #1}_{\scriptstyle #2}$}
		\text{$#1\,$\raisebox{-1.4mm}{$\smash{\vrule height 9.5pt depth 0\dp1}\,\scriptstyle #2$}}}
	{\setbox1\hbox{${\textstyle #1}_{\scriptstyle #2}$}
		\text{$#1\,$\raisebox{-1.4mm}{$\smash{\vrule height 9.5pt depth 0\dp1}\,\scriptstyle #2$}}}
	{\setbox1\hbox{${\scriptstyle #1}_{\scriptscriptstyle #2}$}
		\text{$#1\,$\raisebox{-1mm}{$\smash{\vrule height 6pt depth 0\dp1}\,\scriptstyle #2$}}}
	{\setbox1\hbox{${\scriptscriptstyle #1}_{\scriptscriptstyle #2}$}
		\text{$#1\,$\raisebox{-1mm}{$\smash{\vrule height 5.5pt depth 0\dp1}\,\scriptstyle #2$}}}}

\DeclareMathOperator{\Pfaff}{Pf}
\newcommand{\ap}{\mathrm{ap\,}}
\newcommand{\open}{\overset{\resizebox{1.1mm}{1.1mm}{$\circ$}}}
\DeclareMathOperator{\tr}{Tr}

\makeatletter
\newcommand*{\mint}[1]{%
	\mint@l{#1}{}%
}
\newcommand*{\mint@l}[2]{%
	\@ifnextchar\limits{%
		\mint@l{#1}%
	}{%
	\@ifnextchar\nolimits{%
		\mint@l{#1}%
	}{%
	\@ifnextchar\displaylimits{%
		\mint@l{#1}%
	}{%
	\mint@s{#2}{#1}%
}%
}%
}%
}
\newcommand*{\mint@s}[2]{%
	\@ifnextchar_{%
		\mint@sub{#1}{#2}%
	}{%
	\@ifnextchar^{%
		\mint@sup{#1}{#2}%
	}{%
	\mint@{#1}{#2}{}{}%
}%
}%
}
\def\mint@sub#1#2_#3{%
	\@ifnextchar^{%
		\mint@sub@sup{#1}{#2}{#3}%
	}{%
	\mint@{#1}{#2}{#3}{}%
}%
}
\def\mint@sup#1#2^#3{%
	\@ifnextchar_{%
		\mint@sub@sup{#1}{#2}{#3}%
	}{%
	\mint@{#1}{#2}{}{#3}%
}%
}
\def\mint@sub@sup#1#2#3^#4{%
	\mint@{#1}{#2}{#3}{#4}%
}
\def\mint@sup@sub#1#2#3_#4{%
	\mint@{#1}{#2}{#4}{#3}%
}
\newcommand*{\mint@}[4]{%
	\mathop{}%
	\mkern-\thinmuskip
	\mathchoice{%
		\mint@@{#1}{#2}{#3}{#4}%
		\displaystyle\textstyle\scriptstyle
	}{%
	\mint@@{#1}{#2}{#3}{#4}%
	\textstyle\scriptstyle\scriptstyle
}{%
\mint@@{#1}{#2}{#3}{#4}%
\scriptstyle\scriptscriptstyle\scriptscriptstyle
}{%
\mint@@{#1}{#2}{#3}{#4}%
\scriptscriptstyle\scriptscriptstyle\scriptscriptstyle
}%
\mkern-\thinmuskip
\int#1%
\ifx\\#3\\\else_{#3}\fi
\ifx\\#4\\\else^{#4}\fi  
}

\newcommand*{\mint@@}[7]{%
	\begingroup
	\sbox0{$#5\int\m@th$}%
	\sbox2{$#5\int_{}\m@th$}%
	\dimen2=\wd0 %
	\let\mint@limits=#1\relax
	\ifx\mint@limits\relax
	\sbox4{$#5\int_{\kern1sp}^{\kern1sp}\m@th$}%
	\ifdim\wd4>\wd2 %
	\let\mint@limits=\nolimits
	\else
	\let\mint@limits=\limits
	\fi
	\fi
	\ifx\mint@limits\displaylimits
	\ifx#5\displaystyle
	\let\mint@limits=\limits
	\fi
	\fi
	\ifx\mint@limits\limits
	\sbox0{$#7#3\m@th$}%
	\sbox2{$#7#4\m@th$}%
	\ifdim\wd0>\dimen2 %
	\dimen2=\wd0 %
	\fi
	\ifdim\wd2>\dimen2 %
	\dimen2=\wd2 %
	\fi
	\fi
	\rlap{%
		$#5%
		\vcenter{%
			\hbox to\dimen2{%
				\hss
				$#6{#2}\m@th$%
				\hss
			}%
		}%
		$%
	}%
	\endgroup
}

\begin{document}

\title{Spectral Multipliers on $2$-Step Stratified Groups, I
}

\date{}

\author{Mattia Calzi\thanks{The author is partially supported by the grant PRIN 2015 \emph{Real and Complex Manifolds:  Geometry, Topology and Harmonic Analysis}, and is member of the Gruppo Nazionale per l’Analisi Matematica, la Probabilità e le loro Applicazioni (GNAMPA) of the Istituto Nazionale di Alta Matematica (INdAM).}
}

\theoremstyle{definition}
\newtheorem{definition}{Definition}[section]

\newtheorem{remark}[definition]{Remark}

\theoremstyle{plain}
\newtheorem{theorem}[definition]{Theorem}

\newtheorem{lemma}[definition]{Lemma}

\newtheorem{proposition}[definition]{Proposition}

\newtheorem{corollary}[definition]{Corollary}

\maketitle

\begin{small}
	\section*{Abstract}
	Given a $2$-step stratified group which does not satisfy a slight strengthening of the Moore-Wolf condition, a sub-Laplacian $\Lc$ and a family $\Tc$ of elements of the derived algebra, we study the convolution kernels associated with the operators of the form $m(\Lc, -i \Tc)$. 
	Under suitable conditions, we prove that: i) if the convolution kernel of the operator $m(\Lc,-i \Tc)$ belongs to $L^1$, then $m$ equals almost everywhere a continuous function vanishing at $\infty$ (`Riemann-Lebesgue lemma'); ii) if the convolution kernel of the operator $m(\Lc,-i\Tc)$ is a Schwartz function, then $m$ equals almost everywhere a Schwartz function.
\end{small}

\section{Introduction}
\label{intro}

Let $\Lc$ be a translation-invariant differential operator on $\R^n$. In many situations, the study of $\Lc$ may be simplified by means of the Fourier transform, since $\Fc \Lc$ is a polynomial.
If we consider a left-invariant differential operator $\Lc$ on a Lie group $G$, we may still study $\Lc$ by means of the Fourier transform; however, if $G$ is not commutative, the Fourier transform is less manageable than in the commutative case, so that a different approach is preferable. 

A reasonable alternative is provided by  the spectral theorem. 
However, an approach of this kind is very sensitive to the operators involved, while the Fourier transform allows, in principle, to treat all the left-invariant differential operators at the same time. 
For this reason, it might be sensible to consider a finite family of commuting operators  $\Lc_1,\dots, \Lc_k$ instead of a single one. 

Now, assume that $\Lc_1,\dots, \Lc_k$ are formally self-adjoint left-invariant differential operators on $G$, each of which induces an essentially self-adjoint operator on $C^\infty(G)$; assume that the self-adjoint operators induced by $\Lc_1,\dots, \Lc_k$ commute. Then, there is a unique spectral measure $\mi$ on $\R^k$ such that 
\[
\Lc_j \phi=\int_{\R^k} \lambda_j \,\dd \mi(\lambda) \,\phi
\]
for every $\phi \in C^\infty_c(G)$. If $m\colon \R^k\to \C$ is bounded and $\mi$-measurable, we may then associate with $m$ a distribution $\Kc(m)$ such that
\[
m(\Lc_1,\dots, \Lc_k) \, \phi= \phi*\Kc(m)
\]
for every  $\phi \in C^\infty_c(G)$. The mapping $\Kc$ is the desired substitute for the (inverse) Fourier transform.

One may then investigate the similarities between $\Kc$ and the (inverse) Fourier transform. For instance, one may consider the following questions:
\begin{itemize}
	\item does the `Riemann-Lebesgue' property hold? In other words, if $m\in L^\infty(\mi)$ and $\Kc(m)\in L^1(G)$, does $m$ necessarily admit a continuous representative?	
	
	\item is there a positive Radon measure $\beta$ on $\R^k$ such that $\Kc$ extends to an isometry of $L^2(\beta)$ into $L^2(G)$?
	
	\item if such a `Plancherel measure' $\beta$ exists, is it possible to find an `integral kernel' $\chi\in L^1_\loc(\beta \otimes \nu_G)$\footnote{Here, $\nu_G$ denotes a fixed (left or right) Haar measure on $G$. } such that, for every $m\in L^\infty(\beta)$ with compact support,
	\[
	\Kc(m)(g)=\int_{\R^k} m(\lambda)\chi(\lambda,g)\,\dd \beta(\lambda)
	\] 
	for almost every $g\in G$?
	
	\item if $G$ is a group of polynomial growth, so that $\Sc(G)$ can be defined in a reasonable way, does $\Kc$ map $\Sc(\R^k)$ into $\Sc(G)$?
	
	\item if $G$ is a group of polynomial growth and $\Kc(m)\in \Sc(G)$ for some $m\in L^\infty(\mi)$, does $m$ necessarily admit a representative in $\Sc(\R^k)$?
\end{itemize}

Some of these questions have already been addressed in certain situations.
In the case of one operator, the construction of a Plancherel measure dates back to M.\ Christ~\cite[Proposition 3]{Christ} for the case of a homogeneous sub-Laplacian on a stratified group. The `integral kernel' was then introduced by L.\ Tolomeo~\cite[Theorem 2.11]{Tolomeo} for a sub-Laplacian on a group of polynomial growth.
Further, A.\ Hulanicki~\cite{Hulanicki} showed that Schwartz multipliers have Schwartz kernels in the setting of a positive Rockland operator on a graded group. On the other hand, a positive answer to the last question is only known for homogeneous sub-Laplacians on stratified groups and for sub-Laplacians on the plane motion group~\cite{MartiniRicciTolomeo}.

Concerning the case of more operators, Hulanicki's theorem was extended first for some families on the Heisenberg group by
A.\ Veneruso~\cite{Veneruso}, and then to the case of a weighted subcoercive system of operators on a group of polynomial growth by A.\ Martini~\cite[Proposition 4.2.1]{Martini}. 
A.\ Martini~\cite[Theorem 3.2.7]{Martini} also extended the existence result for the Plancherel measure to the case of a weighted subcoercive system of operators on a general Lie group.
As for what concerns the correspondence between Schwartz kernels and Schwartz multipliers, it has been proved on a large class of nilpotent groups for commuting families of differential operators that are invariant under the action of a compact group (nilpotent Gelfand pairs); see the work of F.\ Astengo, B.\ Di Blasio and F.\ Ricci~\cite{AstengoDiBlasioRicci,AstengoDiBlasioRicci2}, V.\ Fischer and F.\ Ricci~\cite{FischerRicci}, V.\ Fischer, F.\ Ricci and O.\ Yakimova~\cite{FischerRicciYakimova}.

\smallskip

In this paper, we focus our attention on two properties, which we call $(RL)$ (`Riemann-Lebesgue') and $(S)$ (`Schwartz'); these properties correspond to the first and fifth questions, respectively. 

The case of a homogeneous operator on a homogeneous group is greatly simplified by the presence of the dilations, to the point that property $(RL)$ holds automatically (cf.~Theorems~\ref{teo:2} and~\ref{teo:3}), while property $(S)$ can be characterized in a simple way, at least on abelian groups; we shall present this characterization in a future paper.
On the other hand, the case of more operators, even homogeneous, is much more involved and properties $(RL)$ and $(S)$ may fail even in standard situations like abelian groups or the Heisenberg groups. 
We shall present examples of these pathological behaviours in a future paper.

In the first part of the paper, we introduce Rockland families on homogeneous groups, and some relevant objects such as the `kernel transform' $\Kc$, the `Plancherel measure' $\beta$, the `integral kernel' $\chi$, and the `multiplier transform' $\Mc$ (Section~\ref{sec:3}).  
Then, we discuss the possibility of transferring properties $(RL)$ and $(S)$ to products of groups (Section~\ref{sez:7}) or to image families under polynomial maps (Section~\ref{sec:8}). 
While the former case is relatively simple, the latter one is more delicate; in Sections~\ref{sec:6} and~\ref{sec:7} we prove some general results which can be used to prove the validity of property $(RL)$ or $(S)$ for an image family once the validity of the corresponding property for the `main family' is known.  

In the second part of the paper we focus on the case of sub-Laplacians and elements of the centre of $\gf$ on a $2$-step stratified group $G$. 
Even in this specific context, there is a wide variety of situations. In particular, we distinguish two classes of such groups, where the families of the preceding kind behave quite differently:
\begin{itemize}
	\item the groups $G$ which have a homogeneous subgroup $G'$ contained in $[G,G]$ such that the quotient of $G$ by $G'$ is a Heisenberg group;
	
	\item the groups $G$ which have no such quotients.
\end{itemize}
We call the groups of the first kind $MW^+$ groups, or groups satisfying the $MW^+$ condition, since the condition which defines these groups is a slight strengthening of the Moore-Wolf condition (cf.~\cite{MooreWolf} and also~\cite{MullerRicci}); in fact, the condition that was actually considered in~\cite{MooreWolf} is related to the centre $Z$ of $G$ instead of $[G,G]$. Nevertheless, one may always factor out an abelian group so as to reduce to a group with $Z=[G,G]$. 
In addition, the exposition becomes clearer if we consider only elements of the derived algebra, instead of the centre, so that the above dichotomy becomes more natural; thanks to Remark~\ref{oss:2}, we can always reduce to this case.
Since the treatment of these two classes of groups is quite different, we focus here on groups which do \emph{not} satisfy the $MW^+$ condition; we shall study $MW^+$ groups in a future paper.

In Section~\ref{sec:11}, we give an expression for the Plancherel measure; we also provide the reader with the tools needed to find an expression for the integral kernel, which we do not display explicitly since we do not consider it as particularly illuminating.
Finally, in the last two sections we shall prove several conditions under which properties $(RL)$ and $(S)$ hold for the given families.

\section{Definitions and Notation}

A homogeneous group is a simply connected nilpotent Lie group $G$ endowed with a family of dilations which are group automorphisms; we shall denote them by $r\cdot x$ ($r>0$, $x\in G$).
The homogeneous dimension $Q$ of $G$ is the sum of the homogeneous degrees of the elements of any homogeneous basis of the Lie algebra of $G$.
We denote by $\nu_G$ a Haar measure on $G$, which is the image of a fixed Lebesgue measure on the Lie algebra under the exponential map.

A homogeneous norm on $G$ is a proper mapping $\abs{\,\cdot\,}\colon G\to \R_+$ which is symmetric and homogeneous of homogeneous degree $1$.

\begin{definition}
	A differential operator $X$ on $G$ is homogeneous of degree $d\in \C$ if
	\[
	X [ \phi(r\,\cdot\,) ]=r^d (X \phi)(r\,\cdot\,) 
	\]
	for every $\phi \in C^\infty(G)$ and for every $r>0$.
\end{definition}

We end this section with some general notation concerning measures.
First of all, unless explicitly stated, all measures are supposed to be positive and Radon. 
For the sake of simplicity, we only deal with Radon measures on Polish spaces, that is, topological spaces with a countable base whose topology is induced by a complete metric.
For example, every locally compact space with a countable base is a Polish space, but we shall need to deal with some Polish spaces which are \emph{not} locally compact (cf.~the proof of Theorem~\ref{prop:10}).
If $X$ is a Polish space, then a positive Borel measure $\mi$ on $X$ is a Radon measure if and only if it is locally finite (cf.~\cite[Theorem 2 and Proposition 3 of Chapter IX, § 3]{BourbakiInt2}).

Now, if $X$ and $Y$ are Polish spaces, $\mi$ is a Radon measure on $X$, and  $\pi\colon X\to Y$ is a $\mi$-measurable mapping, then $\pi$ is called $\mi$-proper if $\pi_*(\mi)$ is a Radon measure.
Observe that, if $\pi$ is proper, then it is $\mi$-proper (cf.~\cite[Remark 2 of Chapter IX, § 2, No.\ 3]{BourbakiInt2}).

If $\mi$ is a Radon measure on a Polish space $X$, and $f\in L^1_\loc(\mi)$, then we shall denote by $f\cdot \mi$ the Radon measure $E\mapsto\int_E f\,\dd \mi$.
We say that two positive Radon measures on a Polish space are equivalent if they share the same negligible sets; in other words, if they are absolutely continuous with respect to one another.

If $X$ is a locally compact space, then we denote by $C_0(X)$ the space of complex-valued continuous functions on $X$ which vanish at $\infty$, endowed with the maximum norm. We denote by $\Mc^1(X)$ the dual of $C_0(X)$, that is, the space of bounded (Radon) measures on $X$.

\section{Rockland Families and the Kernel Transform}
\label{sec:3}

In this section, $G$ denotes a homogeneous group of dimension $n$ and homogeneous dimension~$Q$.

\begin{definition}
	Let $\Lc_A=(\Lc_\alpha)_{\alpha\in A}$ be a family of differential operators on $G$.
	We say that $\Lc_A$ is jointly hypoelliptic if the following hold: if $V$ is an open subset of $G$ and $T$ is a distribution on $V$ such that $\Lc_\alpha T\in C^\infty(V)$ for every $\alpha\in A$, then $T\in C^\infty(V)$.
\end{definition}

The following result enriches~\cite[Proposition 3.6.3]{Martini}.

\begin{theorem}\label{teo:1}
	Let $\Lc_A=(\Lc_\alpha)_{\alpha\in A}$ be a non-empty finite family of formally self-adjoint, homogeneous, left-invariant differential operators without terms of order $0$ on $G$. Then, the following conditions are equivalent:
	\begin{enumerate}
		\item $\Lc_A$ is jointly hypoelliptic;
		
		\item for every continuous non-trivial irreducible unitary representation $\pi$ of $G$ in a hilbertian space $H$, the family of operators $\dd \pi(\Lc_A)$ is jointly injective on $C^\infty(\pi)$;
		
		\item the (non-unital) algebra generated by $\Lc_A$ contains a positive Rockland operator, possibly with respect to a different family of dilations on $G$ with respect to which the $\Lc_\alpha$ are still homogeneous.
	\end{enumerate} 
	Assume, in addition, that the $\Lc_\alpha$ commute as differential operators. Then, the preceding conditions are equivalent to the following one:
	\begin{enumerate}
		\item[4.] the $\Lc_\alpha$ are essentially self-adjoint on $C^\infty_c(G)$, their self-adjoint extensions commute and for every $m\in \Sc(\R^A)$ the convolution kernel of the operator $m(\Lc_A)$ belongs to $\Sc(G)$. 
	\end{enumerate}
\end{theorem}

\begin{proof}
	{\bf 1 $\implies$ 2.} This is a simple adaptation of the proof of~\cite[Theorem 1]{Beals}.

	{\bf2 $\implies$ 3.} This is the implication $(ii) \implies (i)$ of~\cite[Proposition 3.6.3]{Martini}.
	
	{\bf 3 $\implies$ 1.} Take an open subset $V$ of $G$ and $T\in \Dc'(T)$ such that $\Lc_\alpha T$ is $C^\infty$ on $V$ for every $\alpha\in A$. Take $P\in \C[A]$ such that $P(0)=0$ and $P(\Lc_A)$ is hypoelliptic. Then,  $P(\Lc_A)T$ is $C^\infty$ on $V$, so that $T$ is $C^\infty$ on $V$.
	
	\medskip
	
	Now, assume that the $\Lc_\alpha$ commute as differential operators.
	
	{\bf3 $\implies$ 4.} This follows from~\cite[Propositions 1.4.4, 3.1.2, and 4.2.1]{Martini}.
	
	{\bf4 $\implies$ 3.} Notice first that, by~\cite[Proposition 1.1]{Miller}, we may replace the family of dilations of $G$ with another one in such a way that  the $\Lc_\alpha$ are still homogeneous and the degrees of homogeneity $\delta_\alpha$ of the $\Lc_\alpha$ all belong to $\N^*$.\footnote{If $\Lc_{\alpha}=0$, choose $\delta_{\alpha}=1$.} 
	Then, define $k_\alpha\coloneqq \prod_{\alpha'\neq \alpha} \delta_{\alpha'}$ for every $\alpha\in A$, and  $P(X_A)\coloneqq\sum_{\alpha\in A} X_\alpha^{2 k_{\alpha} }$; observe that $P$ defines a positive \emph{homogeneous} proper polynomial mapping on $\R^A$, of homogeneous degree $\delta'=2 \prod_{\alpha\in A} \delta_\alpha$.	
	Now, take $t\Meg 0$ and let $p_t$ be the convolution kernel of the operator $e^{- t P(\Lc_A)}$; by assumption, $p_t\in \Sc(G)$ for $t>0$, while $p_0=\delta_e$. In addition, it is readily seen that 
	\[
	p_t(g)=t^{-\sfrac{Q}{\delta'}} p_1\left( t^{-\sfrac{1}{\delta'}}\cdot g  \right)
	\]
	for every $t>0$ and for every $g\in G$. Now, define
	\[
	p(t,g)\coloneqq \begin{cases}
	p_t(g) & \text{if $t>0$}\\
	0 & \text{if $t\meg 0$}.
	\end{cases}
	\]
	Then, it is easily seen that $p$ is of class $C^\infty$ on $(\R\times G)\setminus \Set{(0,e)}$. In addition, for every $\phi \in C^\infty_c(G)$ the mapping $t \mapsto \phi* p_t\in L^2(G)$ is continuous on $\R_+$ and differentiable on $\R_+^*$, with derivative $t\mapsto \phi*[P(\Lc_A) p_t]$. By the arbitrariness of $\phi$, we deduce that the mapping $t\mapsto p_t\in \Dc'(G)$ is of class $C^1$ on $\R_+$, with derivative $t\mapsto P(\Lc_A) p_t$ (cf.~\cite[Theorem 3.1]{DixmierMalliavin}). Hence, by means of routine arguments we see that $p$ is a fundamental solution of the heat operator $\partial_t-P(\Lc_A)$ on $\R\times G$. 
	Since $\partial_t-P(\Lc_A)$ is formally self-adjoint, we see that $\check{\overline{p}}$ is a fundamental solution of the right-invariant differential operator associated with $\partial_t-P(\Lc_A)$.	
	Arguing as in the proof of~\cite[Theorem 2.1]{Treves2}, we see that $\partial_t-P(\Lc_A)$ is hypoelliptic, so that also $P(\Lc_A)$ is hypoelliptic.
\end{proof}

\begin{definition}
	A Rockland family is a non-empty finite \emph{commutative} family of homogeneous left-invariant differential operators without terms of order $0$ which satisfies the equivalent conditions of Theorem~\ref{teo:1}. 
\end{definition}

Here we depart slightly from the notion of `Rockland system' as defined in~\cite{Martini}.
Indeed, a Rockland system is a Rockland family, while a Rockland family need not be a Rockland system, since the algebra it generates need not contain a Rockland operator. 
Nevertheless, the difference is only illusory: as Theorem~\ref{teo:1} shows, given a Rockland family $\Lc_A$, one may change the dilations of $G$ in such a way that $\Lc_A$ becomes a Rockland system.
In other words, up to a change of dilations, there is no difference between Rockland families and Rockland systems.

Notice that, as a consequence of the results of Section~\ref{sec:8}, the properties we are going to investigate do not pertain to the chosen family $\Lc_A$, but actually to the (non-unital) algebra it generates.
As a matter of fact, we can start with a commutative, finitely generated, formally self-adjoint and dilation-invariant sub-algebra of $\Uf_\C(\gf)$ and require that its elements have no constant terms and that it contains a hypoelliptic operator without constant terms.
It is not hard to see that such  algebras are generated by a Rockland family (use~\cite{RobbinSalamon} to prove that dilation-invariant sub-algebras are graded, that is, generated by homogeneous elements), and that different Rockland families which generate the same algebra are equivalent in a natural sense.

Notice, in addition, that we do not impose any minimality conditions on the chosen family, in terms of the aforementioned equivalence; we do not even require that each $\Lc_\alpha$ should be non-zero.
This choice does not provide serious inconveniences; instead, it makes the exposition simpler, since we do not have to check at each step that the families we introduce are `minimal' (cf., for instance, Proposition~\ref{prop:3}).

\begin{definition}\label{def:1}
	Let $\Lc_A$ be a Rockland family. Then, we denote by $\mi_{\Lc_A}$ the spectral measure associated with the self-adjoint extensions of the $\Lc_\alpha$.
	
	We say that a $\mi_{\Lc_A}$-measurable function $m\colon \R^A\to \C$ admits a kernel if $C^\infty_c(G)$ is contained in the domain of $m(\Lc_A)$. In this case, $\Sc(G)$ is contained in the domain of $m(\Lc_A)$ and there is a unique $K\in \Sc'(G)$ such that $m(\Lc_A)\phi= \phi* K$ for every $\phi \in \Sc(G)$ (cf.~\cite[Theorem 7.2]{DixmierMalliavin}); we shall denote $K$ by $\Kc_{\Lc_A}(m)$.
\end{definition}

\begin{definition}
	Let $\Lc_A$ be a Rockland family. We shall say that $\Lc_A$ satisfies property:
	\begin{itemize}
		\item[$(RL)$] (`Riemann-Lebesgue') if every $m\in L^\infty(\mi_{\Lc_A})$ such that $\Kc_{\Lc_A}(m)\in L^1(G)$ has a continuous representative;
		
		\item[$(S)$] (`Schwartz') if every $m\in L^\infty(\mi_{\Lc_A})$ such that $\Kc_{\Lc_A}(m)\in \Sc(G)$ has a representative in $\Sc(\R^A)$.
	\end{itemize}
\end{definition}

\begin{remark}\label{oss:1:1}
	Observe that we did not require that $m$ has a representative in $C_0(\R^A)$ in the definition of property $(RL)$. Actually, the fact that $m$ vanishes at $\infty$ is basically automatic (cf.~\cite[Proposition 3.2.11]{Martini}). 
\end{remark}

Notice that, thanks to~\cite[Proposition 3.6.3]{Martini},  we may take advantage of the study of `weighted subcoercive systems' pursued in~\cite{Martini}. Actually, many of the general results proved below hold for weighted subcoercive systems on (say) groups of polynomial growth.
In particular, we shall often use without reference some elementary properties of $\Kc_{\Lc_A}$. For example, 
\[
\Kc_{\Lc_A}(\overline m)=\Kc_{\Lc_A}(m)^*= \overline{[\Kc_{\Lc_A}(m)]\check\;}.
\]
We shall also have to deal with equalities of the form
\[
\Kc_{\Lc_A}(m_1 m_2)=\Kc_{\Lc_A}(m_1)*\Kc_{\Lc_A}(m_2).
\]
We leave to the reader to verify case by case that such equalities hold (or else to refer to~\cite{Martini} whenever possible).

Before we proceed, let us state a couple of useful results. 
The first one is basically a corollary of~\cite[Proposition 3.2.4]{Martini}.

\begin{proposition}\label{prop:3}
	Let $G,G'$ be two non-trivial homogeneous groups, and $\pi$ a \emph{homogeneous} homomorphism of $G$ \emph{onto} $G'$. Let $\Lc_A$ be a Rockland family on $G$. Then, the following hold:
	\begin{enumerate}
		\item $\dd \pi(\Lc_A)=(\dd \pi(\Lc_\alpha))_{\alpha\in A}$ is a Rockland family on $G'$;
		
		\item $\sigma(\dd \pi(\Lc_A))\subseteq \sigma(\Lc_A)$;
		
		\item if $m\colon E_{\Lc_A}\to \C$ is $\beta_{\Lc_A}$-measurable and continuous on an open set which carries $\beta_{\dd \pi(\Lc_A)}$, and if $\Kc_{\Lc_A}(m)\in \Mc^1(G)+\Ec'(G)$, then
		\[
		\pi_*(\Kc_{\Lc_A}(m))= \Kc_{\dd \pi(\Lc_A)}(m).
		\]
	\end{enumerate}
\end{proposition}

\begin{proof}
	{\bf1.} The fact that $\dd \pi(\Lc_A)$ is Rockland follows from the fact that, if $\widetilde \pi$ is a continuous unitary representation of $G'$, then $\widetilde \pi \circ \pi$ is a continuous unitary representation of $G$, with $C^\infty(\widetilde \pi)=C^\infty(\widetilde \pi\circ \pi)$ since $\pi$ is a submersion, and $\dd \widetilde \pi(\dd \pi(\Lc_A))= \dd (\widetilde \pi \circ \pi)(\Lc_A)$; finally, $\widetilde \pi\circ \pi$ is irreducible or trivial if and only if $\widetilde \pi$ is irreducible or trivial, respectively.
	
	{\bf3.} Let $\widetilde \pi$ be the right quasi-regular representation of $G$ in $L^2(G')$, that is, $\widetilde \pi(g) f=f(\,\cdot\,\pi(g))$ for every $g\in G$ and for every $f\in L^2(G')$. Then~\cite[Proposition 3.2.4]{Martini}, applied to $\widetilde \pi$, implies that our assertion holds if $m\in C_0(E_{\Lc_A})$ and $\Kc_{\Lc_A}(m)\in L^1(G)$. The general case follows by approximation.
	
	{\bf2.}  This follows easily from~{\bf3}.
\end{proof}

The following definition will shorten the notation in the sequel.
\begin{definition}
	Let $F$ be a subspace of $\Dc'(G)$. Then, we  denote by $F_{\Lc_A}$ the set of $\Kc_{\Lc_A}(m)$ as $m$ runs through the set of $\mi_{\Lc_A}$-measurable functions which admit a kernel in $F$.
\end{definition}

\begin{proposition}\label{prop:1}
	Let $F$ be a Fréchet space which is continuously embedded in $\Mc^1(G)$ or, more generally, in the right convolutors of $L^2(G)$. Then, $F_{\Lc_A}$ is closed in $F$.
\end{proposition}

In particular, this applies to $L^1(G)$ and $\Sc(G)$. With more effort, one may generalize this result to any locally convex space $F$ which is continuously embedded in $\Dc'(G)$ and for which the bilinear mapping $*\colon  C_c^\infty(G)\times F\to L^2(G)$ is separately continuous.

\begin{proof}
	Indeed, let $(m_j)$ be a sequence in $L^\infty(\mi_{\Lc_A})$ such that the sequence $(\Kc_{\Lc_A}(m_j))$ converges to some $f$ in $F$. 
	Then, $(m_j(\Lc_A))$ is a Cauchy sequence in $\Lc(L^2(G))$, so that $(m_j)$ is a Cauchy sequence in $L^\infty(\mi_{\Lc_A})$ by spectral theory. Therefore, it converges to some $m$ in $L^\infty(\mi_{\Lc_A})$; hence, $\Kc_{\Lc_A}(m_j)$ converges to $\Kc_{\Lc_A}(m)$ in $\Sc'(G)$ (cf.~\cite[Theorem 7.2]{DixmierMalliavin}). Therefore, $\Kc_{\Lc_A}(m)=f$.
\end{proof}

We shall often need some dilations on $\R^A$ which reflect the homogeneity of the $\Lc_\alpha$. This leads to the following definition.

\begin{definition}
	Let $\Lc_A$ be a Rockland family. Then, we shall define $E_{\Lc_A}$ as $\R^A$, endowed with the dilations
	\[
	r \cdot (x_\alpha)\coloneqq (r^{\delta_\alpha} x_\alpha)
	\]
	for every $r>0$ and for every $(x_\alpha)\in \R^A$; here, $\delta_\alpha$ is the homogeneous degree of $\Lc_\alpha$ if $\Lc_\alpha\neq 0$, while $\delta_\alpha=1$ otherwise. We denote by $\abs{\,\cdot\,}$ a homogeneous norm on $E_{\Lc_A}$.
\end{definition}

The following proposition is basically a consequence of~\cite[Theorem 3.2.7 and Proposition 3.6.1]{Martini}. The proof is omitted. 

\begin{theorem}
	Let $\Lc_A$ be a Rockland family. Then, there is a unique positive Radon measure $\beta_{\Lc_A}$ on $E_{\Lc_A}$ such that the following hold:
	\begin{enumerate}
		\item $\mi_{\Lc_A}$ and $\beta_{\Lc_A}$ are equivalent;
		
		\item $\Kc_{\Lc_A}$ induces an isometry of $L^2(\beta_{\Lc_A})$ onto $L^2_{\Lc_A}(G)$;
		
		\item $(r\,\cdot \,)_*(\beta_{\Lc_A})=r^{-Q} \beta_{\Lc_A}$ for every $r>0$.
	\end{enumerate}
\end{theorem}

The following corollary has already been considered in~\cite[Proposition 3.2.12]{Martini} for the case $p=1$. The general case follows by interpolation.

\begin{corollary}
	Take $p\in [1,2]$. Then, $\Kc_{\Lc_A}$ induces a unique continuous linear mapping
	\[
	\Kc_{\Lc_A,p}\colon L^p(\beta_{\Lc_A})\to L^{p'}(G).
	\]
	
	In addition, $\Kc_{\Lc_A,1}$ maps $L^1(\beta_{\Lc_A})$ into $C_0(G)$,  has norm $1$ and induces an isometry from the set of positive $\beta_{\Lc_A}$-integrable functions into the set of continuous functions of positive type on $G$.
\end{corollary}

From the preceding corollary we deduce the existence of an `integral kernel' $\chi_{\Lc_A}$ for the `kernel transform' $\Kc_{\Lc_A}$. This integral kernel was introduced in~\cite[Theorem 2.11]{Tolomeo} for a sub-Laplacian on a group of polynomial growth; since many results of the remainder of this section are basically extensions of those presented in~\cite{Tolomeo}, we shall omit most of the proofs.

\begin{proposition}
	There is a unique $\chi_{\Lc_A}\in L^\infty(\beta_{\Lc_A}\otimes \nu_G)$ such that
	\[
	\Kc_{\Lc_A}(m)(g)=\int_{E_{\Lc_A}} m(\lambda) \chi_{\Lc_A}(\lambda,g) \,\dd \beta_{\Lc_A}(\lambda)
	\]
	for $\nu_G$-almost every $g\in G$. 
	In addition, $\norm{\chi_{\Lc_A}}_\infty=1$.
\end{proposition}

The proof follows the lines of the proof of~\cite[Theorem 2.11]{Tolomeo}. One may also make use of the Dunford-Pettis Theorem.

We now pass to show some of the main properties of $\chi_{\Lc_A}$. In particular, we shall find some representatives of $\chi_{\Lc_A}$ which are particularly well-behaved. 
The following simple result generalizes~\cite[Theorem 2.33]{Tolomeo}.

\begin{proposition}\label{prop:3:7}
	For every $s>0$ and for $(\beta_{\Lc_A}\otimes \nu_G)$-almost every $(\lambda,g)\in E_{\Lc_A}\times G$,
	\[
	\chi_{\Lc_A}(s\cdot\lambda,g)=\chi_{\Lc_A}(\lambda, s \cdot g).
	\]		
\end{proposition}

The following property is reminiscent of an analogous one concerning Gelfand pairs. It extends~\cite[Proposition 2.14]{Tolomeo} to our setting; we shall nevertheless present an alternative proof.

\begin{proposition}\label{prop:3:8}
	Take a $\beta_{\Lc_A}$-measurable function $m\colon E_{\Lc_A}\to \C$ which admits a kernel in $\Mc^1(G)+\Ec'(G)$. Then
	\[
	\Kc_{\Lc_A}(m)*\chi_{\Lc_A}(\lambda,\,\cdot\,)=\chi_{\Lc_A}(\lambda,\,\cdot\,)*\Kc_{\Lc_A}(m)= m(\lambda)\chi_{\Lc_A}(\lambda,\,\cdot\,)
	\]
	for $\beta_{\Lc_A}$-almost every $\lambda\in E_{\Lc_A}$.
\end{proposition}

\begin{proof}
	Notice first that, for every $\phi_2\in C^\infty_c(G)$, the linear functional
	\[
	L^\infty(G)\ni f \mapsto \langle \Kc_{\Lc_A}(m)* f, \phi_2\rangle= \langle f, \Kc_{\Lc_A}(m)\check{\;} * \phi_2\rangle\in \C
	\]
	is continuous with respect to the weak topology $\sigma(L^\infty(G),L^1(G))$. In addition, for every $\phi_1\in C^\infty_c(E_{\Lc_A})$,
	\[
	\Kc_{\Lc_A}(\phi_1)= \int_{E_{\Lc_A}} \phi_1(\lambda)\chi_{\Lc_A}(\lambda,\,\cdot\,)\,\dd \beta_{\Lc_A}(\lambda)
	\]
	in $L^\infty(G)$, endowed with the weak topology $\sigma(L^\infty(G),L^1(G))$. Therefore, 
	\[
	\begin{split}
	&\int_{E_{\Lc_A}} \langle \Kc_{\Lc_A}(m)*\chi_{\Lc_A}(\lambda,\,\cdot\,), \phi_2\rangle\, \phi_1(\lambda)\,\dd \beta_{\Lc_A}(\lambda)= \langle \Kc_{\Lc_A}(m)* \Kc_{\Lc_A}(\phi_1), \phi_2\rangle\\
	&\qquad \qquad\qquad \qquad\qquad \qquad\quad=\langle \Kc_{\Lc_A}(m \phi_1), \phi_2\rangle\\
	&\qquad \qquad\qquad \qquad\qquad \qquad\quad=\int_{E_{\Lc_A}} (m \phi_1)(\lambda) \langle \chi_{\Lc_A}(\lambda,\,\cdot\,), \phi_2\rangle\,\dd \beta_{\Lc_A}(\lambda),
	\end{split}
	\]
	whence the assertion by the arbitrariness of $\phi_2$. The other equality is proved similarly.	
\end{proof}

\begin{corollary}
	Let $P$ be a polynomial on $E_{\Lc_A}$. Then
	\[
	P(\Lc_A)\chi_{\Lc_A}(\lambda,\,\cdot \,)=P(\Lc_A^R)\chi_{\Lc_A}(\lambda,\,\cdot \,)=P(\lambda)\chi_{\Lc_A}(\lambda,\,\cdot\,)
	\]
	for $\beta_{\Lc_A}$-almost every $\lambda\in E_{\Lc_A}$; here, $\Lc_A^R$ denotes the family of right-invariant differential operators which corresponds to $\Lc_A$.
\end{corollary}

In the following result we show the existence of well-behaved representatives of $\chi_{\Lc_A}$. Since it is a straightforward extension of~\cite[Lemmas 2.12 and 2.15 and Propositions 2.17 and 2.18]{Tolomeo}, the proof is omitted.

\begin{theorem}\label{teo:2}
	There is a representative $\chi_{0}$ of $\chi_{\Lc_A}$ such that the following hold:
	\begin{enumerate}
		\item $\chi_{0}(\lambda,\,\cdot\,)$ is a function of positive type of class $C^{\infty}$ with maximum $1$  for every $\lambda\in E_{\Lc_A}$;
		
		\item for every homogeneous left- and right-invariant differential operators $X$ and $Y$ on $G$ of homogeneous degrees $\dd_X$ and $\dd_Y$, respectively, there is a constant $C_{X,Y}>0$ such that
		\[
		\norm{ Y X \chi_{0}(\lambda,\,\cdot\,)  }_\infty\meg C_{\gamma_1,\gamma_2} \abs{\lambda}^{\dd_{X}+\dd_{Y}}
		\]		
		for every $\lambda\in E_{\Lc_A}$;
		
		\item $\chi_{0}(\lambda,\,\cdot\,)$ converges to $\chi_{0}(0,\,\cdot\,)=1$ in $\Ec(G)$ as $\lambda\to 0$; 
		
		\item $\chi_{0}(\,\cdot\,,g)$ is $\beta_{\Lc_A}$-measurable for every $g\in G$.
	\end{enumerate}
\end{theorem}

We conclude this section with some remarks concerning the adjoint of $\Kc_{\Lc_A}$ and the continuity of $\chi_{\Lc_A}$.

\begin{definition}
	We shall denote by $\Mc_{\Lc_A}\colon \Mc^1(G)\to L^\infty(\beta_{\Lc_A})$ the transpose of the mapping 
	\[
	L^1(\beta_{\Lc_A})\ni m\mapsto \Kc_{\Lc_A,1}(m) \check{\;} \in C_0(G).
	\] 
\end{definition}

By the way, $\Mc_{\Lc_A}$ coincides with the adjoint of $\Kc_{\Lc_A}\colon L^2(\beta_{\Lc_A})\to L^2(G)$ on $L^1(G)\cap L^2(G)$. 
Therefore, by interpolation we deduce that $\Mc_{\Lc_A}$ extends to a continuous linear mapping of $L^p(G)$ into $L^{p'}(\beta_{\Lc_A})$ for every $p\in [1,2]$.

The following result extends~\cite[Theorem 2.13]{Tolomeo} to the present setting. The proof is omitted.

\begin{proposition}\label{prop:3:3}
	Take a representative $\chi_{0}$ of $\chi_{\Lc_A}$ as in Theorem~\ref{teo:2}. Then, for every $\mi \in \Mc^1(G)$ we have
	\[
	\Mc_{\Lc_A}(\mi)(\lambda)=\int_{G} \overline{\chi_{0}(\lambda,g) }\,\dd \mi(g)
	\]
	for $\beta_{\Lc_A}$-almost every $\lambda\in E_{\Lc_A}$.
\end{proposition}

\begin{corollary}\label{cor:4}
	Take $m\in L^\infty(\beta_{\Lc_A})$ such that $\Kc_{\Lc_A}(m)\in \Mc^1(G)$ and $\mi \in \Mc^1(G)$. Then
	\[
	\Mc_{\Lc_A}( \Kc_{\Lc_A}(m)*\mi )= \Mc_{\Lc_A}(\mi* \Kc_{\Lc_A}(m) )=m \Mc_{\Lc_A}(\mi).
	\]
\end{corollary}

\begin{proof}
	Indeed, take a representative $\chi_0$ of $\chi_{\Lc_A}$ as in Theorem~\ref{teo:2}. Then, Proposition~\ref{prop:3:8} implies that
	\[
	\begin{split}
	\Mc_{\Lc_A}(  \Kc_{\Lc_A}(m)*\mi )(\lambda)&= \langle \Kc_{\Lc_A}(m)*\mi, \overline {\chi_{0}(\lambda,\,\cdot\,)}\rangle=\langle \mi,  \overline{\Kc_{\Lc_A}(\overline m)*\chi_{0}(\lambda,\,\cdot\,)}\rangle\\
	&= m(\lambda) \langle\mi, \overline{\chi_0(\lambda,\,\cdot\,)}\rangle= m(\lambda) \Mc_{\Lc_A}(\mi)(\lambda) 
	\end{split}
	\]
	for $\beta_{\Lc_A}$-almost every $\lambda\in E_{\Lc_A}$. The other equality is proved analogously.
\end{proof}

\begin{corollary}\label{cor:5}
	Take a function $m\in L^\infty(\beta_{\Lc_A})$ such that $\Kc_{\Lc_A}(m)\in\Mc^1(G)$. Then, $m=\Mc_{\Lc_A}(\Kc_{\Lc_A}(m))$.
	In particular, $m$ is continuous at $0$ and
	\[
	m(0)=\int_{G} \dd \Kc_{\Lc_A}(m).
	\]
\end{corollary}

\begin{proof}
	The first assertion follows from Corollary~\ref{cor:4}, applied with $\mi\coloneqq \delta_e$.
	The second assertion follows from~{\bf3} of Theorem~\ref{teo:2}.
\end{proof}

\begin{theorem}\label{teo:3}
	The following conditions are equivalent:
	\begin{enumerate}
		\item $\chi_{\Lc_A}$ has a representative $\chi_0$ such that $\chi_0(\,\cdot\,, g)$ is continuous on $\sigma(\Lc_A)$ for $\nu_G$-almost every $g\in G$;\footnote{Notice that, in principle, this condition is weaker than separate continuity.}
		
		\item $\Mc_{\Lc_A}$ induces a continuous linear mapping from $L^1(G)$ into $C_0(\sigma(\Lc_A))$;
		
		\item $\Mc_{\Lc_A}$ induces a continuous linear mapping from $\Mc^1(G)$ into $C_b(\sigma(\Lc_A))$;
		
		\item $\chi_{\Lc_A}$ has a continuous representative.
	\end{enumerate}
\end{theorem}

This shows, in particular, that if $\chi_{\Lc_A}$ has a continuous representative, then $\Lc_A$ satisfies property $(RL)$. 
Nevertheless, the converse fails as Remark~\ref{oss:1} shows.

\begin{proof}
	{\bf1 $\implies$ 2.} In order to prove continuity, it suffices to show that
	\[
	\Mc_{\Lc_A}(\phi)(\lambda)=\int_G \overline{\chi_0(\lambda,g)}\,\phi(g)\,\dd g
	\]
	for every $\phi \in L^1(G)$ and for $\beta_{\Lc_A}$-almost every $\lambda\in E_{\Lc_A}$, and to apply the dominated convergence theorem.  In order to prove that $\Mc_{\Lc_A}(\phi)$ vanishes at $\infty$, it suffices to observe that, if $\tau\in C^\infty_c(E_{\Lc_A})$ and $\tau(0)=1$, then $\Mc_{\Lc_A}(\phi)$ is the limit in $C_b(\sigma(\Lc_A))$ of $\Mc_{\Lc_A}(\phi*\Kc_{\Lc_A}(\tau(2^{-j}\,\cdot\,)))$, which equals $\tau(2^{-j}\,\cdot\,) \Mc_{\Lc_A}(\phi)$ by Corollary~\ref{cor:4}.

	{\bf2 $\implies $ 4.} Take $\tau\in \Sc(E_{\Lc_A})$ such that $\tau(\lambda)>0$ for every $\lambda\in E_{\Lc_A}$. Observe that the mapping $G\ni g\mapsto \Kc_{\Lc_A}(\tau)(g\,\cdot\,)\in L^1(G)$ is continuous, so that also the mapping $G\ni g\mapsto \Mc_{\Lc_A}(\Kc_{\Lc_A}(\tau)(g\,\cdot\,))\in C_0(\sigma(\Lc_A))$ is continuous. Therefore, the mapping
	\[
	\sigma(\Lc_A)\times G\ni (\lambda,g)\mapsto \Mc_{\Lc_A}(\Kc_{\Lc_A}(\tau)(g\,\cdot\,))(\lambda) \in \C
	\]
	is continuous. Now, let $\chi_{1}$ be a representative of $\chi_{\Lc_A}$ as in Theorem~\ref{teo:2}. Then, Proposition~\ref{prop:3:8} implies that
	\[
	\begin{split}
	\Mc_{\Lc_A}(\Kc_{\Lc_A}(\tau)(g\,\cdot\,))(\lambda)&= \int_G \Kc_{\Lc_A}(\tau)(g g')\chi_{1}(\lambda,g'^{-1})\,\dd g'= [ \Kc_{\Lc_A}(\tau)*\chi_{1}(\lambda,\,\cdot\,) ](g)= \tau(\lambda) \chi_{1}(\lambda,g)
	\end{split}
	\]
	for $(\beta_{\Lc_A}\otimes \nu_G)$-almost every $(\lambda, g)\in E_{\Lc_A}\times G$. In particular, $\chi_{\Lc_A}$ has a representative which is continuous on $\sigma(\Lc_A)\times G$.  By~\cite[Corollary to Theorem 2 of Chapter IX, § 4, No.\ 3]{BourbakiGT2}, $\chi_{\Lc_A}$ has a continuous representative.
	
	{\bf 4 $\implies$ 1.} Obvious.
	
	{\bf 4 $\implies$ 3.} The proof is similar to that of the implication {\bf1 $\implies$ 2}.
	
	{\bf 3 $\implies$ 2.} This follows from the proof of the  implication {\bf1 $\implies$ 2}.
\end{proof}

\section{Products}\label{sez:7}

In this section we deal with the following situation: we have a finite family of homogeneous groups $(G_A)_{A\in \Ac}$, and on each $G_A$  a Rockland family $\Lc_A$.\footnote{In order to avoid technical problems, we shall assume that the elements of $\Ac$ are pairwise disjoint.} 
Then, we shall consider $G\coloneqq \prod_{A\in \Ac} G_A$, endowed with the dilations
\[
r\cdot (g_A)\coloneqq (r\cdot g_A),
\]
for $r>0$ and $(g_A)\in G$. We shall denote by $ A'$ the union of $\Ac$ and, for every $\alpha\in  A'$, we shall denote by $\Lc'_\alpha$ the operator on $G$ induced by $\Lc_\alpha$. Then,  $\Lc'_{ A'}$ will denote the family $(\Lc'_\alpha)_{\alpha\in   A'}$.
We shall investigate what we can say about $\Lc'_{ A'}$ on the ground of our knowledge of the families $\Lc_A$.
Notice that many of the implications of this section are actually equivalences; nevertheless, we shall leave to the reader the task of stating and proving the easy converses.

The following result is basically a consequence of~\cite[Propositions 3.4.2 and 3.4.3]{Martini}. The proof is omitted.

\begin{proposition}
	The following hold:
	\begin{enumerate}
		\item $\Lc'_{ A'}$ is a Rockland family;
		
		\item take a $\mi_{\Lc_A}$-measurable function $m_A\colon \R^A\to \C$ which admits a kernel for every $A\in \Ac$. Then,  $\bigotimes_{A\in \Ac} m_A$ is $\mi_{\Lc'_{ A'}}$-measurable, admits a kernel, and
		\[
		\Kc_{\Lc'_{ A'}}\left(\bigotimes_{A\in \Ac} m_A  \right)= \bigotimes_{A\in \Ac} \Kc_{\Lc_A}(m_A).
		\]
	\end{enumerate}
\end{proposition}

The following result is basically a consequence of~\cite[Proposition 3.4.4]{Martini}. The proof is omitted.

\begin{proposition}\label{prop:6:2}
	The following hold:
	\begin{enumerate}
		\item $\beta_{\Lc'_{ A'}}= \bigotimes_{A\in \Ac} \beta_{\Lc_A}$;
		
		\item for $(\beta_{\Lc'_{ A'}}\otimes \nu_G)$-almost every $((\lambda_\alpha), (g_A))\in \R^{ A'}\times G$,
		\[
		\chi_{\Lc'_{ A'}} ((\lambda_\alpha)_{\alpha\in  A'}, (g_A)_{A\in \Ac})= \prod_{A\in \Ac} \chi_{\Lc_A}((\lambda_\alpha)_{\alpha\in A}, g_A  ).
		\]
	\end{enumerate}
\end{proposition}

Now we focus on property $(RL)$.  
From now on, we shall sometimes make use of topological tensor products \emph{over $\C$}. We shall generally agree with the notation of~\cite{Treves}, except for the fact that, without further specifications, we shall endow every tensor product with the $\pi$-topology.  

\begin{lemma}\label{lem:1}
	Assume that $\Ac=\Set{A_1,A_2}$. Then, for every $m\in L^1(\beta_{\Lc'_{ A'}})$ and for every $\mi\in \Mc^1(G_{A_2})$ there is $m_{\mi}\in L^1(\beta_{\Lc_{A_1}})$ such that
	\[
	\int_{G_{A_2}}\Kc_{\Lc'_{ A'},1}(m)(\,\cdot\,,g_2)\,\dd \mi(g_2) = \Kc_{\Lc_{A_1},1}(m_{\mi}).
	\]
\end{lemma}

\begin{proof}
	Observe first that 
	\[
	L^1(\beta_{\Lc'_{ A'}})\cong L^1(\beta_{\Lc_{A_1}}; L^1(\beta_{\Lc_{A_2}}))\cong L^1(\beta_{\Lc_{A_1}})\widehat \otimes L^1(\beta_{\Lc_{A_2}})
	\]
	thanks to~\cite[Theorem 46.2]{Treves}. Therefore,~\cite[Theorem 45.1]{Treves} implies that there are $(c_j)\in \ell^1$ and two bounded sequences $(m_{j,1}), (m_{j,2})$ in $L^1(\beta_{\Lc_{A_1}})$ and $L^1(\beta_{\Lc_{A_2}})$, respectively, such that
	\[
	m=\sum_{j\in \N} c_j (m_{j,1}\otimes m_{j,2})
	\]
	in $L^1(\beta_{\Lc'_{ A'}})$. Hence, it suffices to define
	\[
	m_{\mi}\coloneqq\sum_{j\in \N} c_j\, \int_{G_2}\Kc_{\Lc_{A_2},1}(m_{j,2})(g_2)\,\dd \mi(g_2)\, m_{j,1}.
	\]	
\end{proof}

\begin{corollary}\label{cor:2}
	Assume that $\Ac=\Set{A_1,A_2}$. Take $m\in L^\infty(\beta_{\Lc'_{ A'}})$ such that $\Kc_{\Lc'_{ A'}}(m)\in L^1(G)$ and $f\in L^\infty(G_{A_2})$. Then, 
	\[
	\int_{G_{A_2}}\Kc_{\Lc'_{ A'}}(m)(\,\cdot\,,g_2) f(g_2)\,\dd \nu_{G_{A_2}}(g_2)\in L^1_{\Lc_{A_1}}(G_{A_1}).
	\]
	In addition, $\Kc_{\Lc'_{ A'}}(m)(\,\cdot\,,g_2)\in L^1_{\Lc_{A_1}}(G_{A_1})$ for almost every $g_2\in G_{A_2}$.
\end{corollary}

\begin{proof}
	{\bf1.} Assume first that $m$ is compactly supported. Let $(K_j)$ be an increasing sequence of compact subsets of $G_{A_2}$ whose union is $G_{A_2}$. Then,
	\[
	\begin{split}
	\lim_{j\to \infty} \int_{K_j} \Kc_{\Lc'_{ A'}}(m)(\,\cdot\,,g_2) f(g_2)\,\dd &\nu_{G_{A_2}}(g_2)=\int_{G_{A_2}} \Kc_{\Lc'_{ A'}}(m)(\,\cdot\,,g_2) f(g_2)\,\dd \nu_{G_{A_2}}(g_2)
	\end{split}
	\]
	in $L^1(G_{A_1})$. The first assertion follows from Lemma~\ref{lem:1} and Proposition~\ref{prop:1}, while the second assertion follows directly from Lemma~\ref{lem:1}.
	
	{\bf2.} Now, take $\tau\in C^\infty_c(E_{\Lc_A})$ such that $\tau(0)=1$, and define $\tau_j\coloneqq \tau(2^{-j}\,\cdot\,)$ for every $j\in \N$. Then,~{\bf1} above implies that 
	\[
	\int_{G_2}\Kc_{\Lc'_{ A'}}(m \tau_j )(\,\cdot\,,g_2) f(g_2)\,\dd \nu_{G_{A_2}}(g_2)\in  L^1_{\Lc_{A_1}}(G_{A_1})
	\]
	and that  $\Kc_{\Lc'_{ A'}}(m \tau_j)(\,\cdot\,,g_2)\in L^1_{\Lc_{A_1}}(G_{A_1})$ for every $j\in \N$ and for almost every $g_2\in G_{A_2}$. 
	Since $\Kc_{\Lc'_{ A'}}(m \tau_j )=\Kc_{\Lc'_{ A'}}(m )*\Kc_{\Lc'_{ A'}}( \tau_j )$ converges to $\Kc_{\Lc'_{ A'}}(m)$ in $L^1(G_{ A'})$, the assertions follow from Proposition~\ref{prop:1}.
\end{proof}

\begin{theorem}\label{teo:4}
	If $\Lc_A$ satisfies property $(RL)$ for every $A\in \Ac$, then $\Lc'_{ A'}$ satisfies property $(RL)$.
\end{theorem}

\begin{proof}
	{\bf1.} Proceeding by induction, we may reduce to the case in which $\Ac=\Set{A_1,A_2}$. In order to simplify the notation, we shall simply write $G_j $ instead of $G_{A_j}$ for $j=1,2$. 
	Now, take $m\in L^\infty(\beta_{\Lc'_{ A'}})$ such that $\Kc_{\Lc'_{ A'}}(m)\in L^1(G)$. Then, Corollary~\ref{cor:5}, Proposition~\ref{prop:6:2} and Fubini's theorem imply that
	\[
	\Mc_{\Lc_{A_1}}[g_1\mapsto\Mc_{\Lc_{A_2}}[\Kc_{\Lc'_{ A'}}(m)(g_1,\,\cdot\,)](\lambda_2)](\lambda_1)=m(\lambda_1, \lambda_2) 
	\]
	for $\beta_{\Lc_{A_1}}$-almost every $\lambda_1\in E_{\Lc_{A_1}}$ and for $\beta_{\Lc_{A_2}}$-almost every $\lambda_2\in E_{\Lc_{A_2}}$.  
	Observe that Lemma~\ref{lem:1} implies that $\Kc_{\Lc'_{ A'}}(m)(g_1,\,\cdot\,)\in  L^1_{\Lc_{A_2}}(G_{2})$ for almost every $g_1\in G_{1}$, and that by assumption $\Mc_{\Lc_{A_2}}$ induces a continuous linear mapping from $L^1_{\Lc_{A_2}}(G_{2}) $ into $C_0(\sigma(\Lc_{A_2}))$. Therefore, the mapping $g_1 \mapsto \Mc_{\Lc_{A_2}}[\Kc_{\Lc'_{ A'}}(m)(g_1,\,\cdot\,)]$ defines an element of $L^1(G_{1}; C_0(\sigma(\Lc_{A_2})))$. 
	
	{\bf2.} Let us prove that, for every $\mi \in \Mc^1(\sigma(\Lc_{A_2}))$,  the mapping 
	\[
		g_1 \mapsto (\mi\Mc_{\Lc_{A_2}})[\Kc_{\Lc'_{ A'}}(m)(g_1,\,\cdot\,)]
	\] 
	belongs to  $L^1_{\Lc_{A_1}}(G_{1})$. Indeed, the preceding considerations show that $\mi\Mc_{\Lc_{A_2}}$ defines an element of $L^1_{\Lc_{A_2}}(G_{2})'$, so that it can be represented by an element of $L^\infty(G_2)$; therefore, the assertion follows from Corollary~\ref{cor:2}. 
	
	Now, let us prove that the mapping
	\[
	\Mc^1(\sigma(\Lc_{A_1}))\ni \mi\mapsto \left[g_1 \mapsto (\mi\Mc_{\Lc_{A_2}})[\Kc_{\Lc'_{ A'}}(m)(g_1,\,\cdot\,)]\right]\in L^1_{\Lc_{A_1}}(G_1)
	\]
	is weakly continuous \emph{on the bounded subsets of $\Mc^1(\sigma(\Lc_{A_1}))$}. Indeed,~\cite[Theorem 46.2]{Treves} implies that $L^1(G_1; C_0(\sigma(\Lc_{A_2})))\cong L^1(G_1)\widehat \otimes C_0(\sigma(\Lc_{A_2}))$, so that~\cite[Theorem 45.1]{Treves} implies that there are $(c_j)\in \ell^1$ and two bounded sequences $(f_{j}), (\phi_{j})$ in $L^1(G_1)$ and $C_0(\sigma(\Lc_{A_2}))$, respectively, such that
	\[
	\left[g_1\mapsto \Mc_{\Lc_{A_2}}[\Kc_{\Lc'_{ A'}}(m)(g_1,\,\cdot\,)]\right]=\sum_{j\in \N} c_j (f_{j}\otimes \phi_{j})
	\]
	in $L^1(G_1; C_0(\sigma(\Lc_{A_2})))$. Since the series 
	\[
	\sum_{j\in \N} c_j \langle \mi, \phi_j\rangle f_j
	\]
	converges uniformly to $g_1 \mapsto (\mi\Mc_{\Lc_{A_2}})[\Kc_{\Lc'_{ A'}}(m)(g_1,\,\cdot\,)] $ as $\mi$ stays in a bounded subset of $\Mc^1(\sigma(\Lc_{A_2}))$, the assertion follows.

	{\bf3.} Next, observe that by assumption $\Mc_{\Lc_{A_1}}$ induces a continuous linear mapping from $L^1_{\Lc_{A_1}}(G_{1})$ into $C_0(\sigma(\Lc_{A_1}))$, so that~{\bf2} above implies that the mapping
	\[
	\begin{split}
	\sigma(\Lc_{A_2})\ni \lambda_2 \mapsto \Mc_{\Lc_{A_1}}\left(g_1\mapsto \Mc_{\Lc_{A_2}}[\Kc_{\Lc'_{ A'}}(m)(g_1,\,\cdot\,)](\lambda_2)   \right)\in C_0(\sigma(\Lc_{A_1}))
	\end{split}
	\] 
	is continuous.
	Therefore, the mapping
	\[
	\sigma(\Lc'_{ A'})\ni (\lambda_1,\lambda_2)\mapsto \Mc_{\Lc_{A_1}}\left(g_1\mapsto \Mc_{\Lc_{A_2}}[\Kc_{\Lc'_{ A'}}(m)(g_1,\,\cdot\,)](\lambda_2)   \right)(\lambda_1)\in \C
	\]
	is continuous; hence, it extends to a continuous mapping $m_0$ on $E_{\Lc'_{ A'}}$ by~\cite[Corollary to Theorem 2 of Chapter IX, § 4, No.\ 3]{BourbakiGT2}.
	Now,~{\bf1} implies that $m_0(\lambda_1,\lambda_2)=m(\lambda_1,\lambda_2)$ for $\beta_{\Lc_{A_1}}$-almost every $\lambda_1\in E_{\Lc_{A_1}}$ and for $\beta_{\Lc_{A_2}}$-almost every $\lambda_2\in E_{\Lc_{A_2}}$. Since both $m$ and $m_0$ are $\beta_{\Lc'_{ A'}}$-measurable, Tonelli's theorem implies that $m=m_0$ $\beta_{\Lc'_{ A'}}$-almost everywhere.
\end{proof}

Now, we focus on property $(S)$. First, we need some definitions.

\begin{definition}
	Let $E$ be a homogeneous group, and let $F$ be a Fréchet space. We shall define $\Sc(E;F)$ as the set of $\phi \in \Ec(E;F)$ such that $(1+\abs{\,\cdot\,})^k X\phi$ is bounded for every $k\in \N$ and for every left-invariant differential operator $X$ on $E$. We shall endow $\Sc(E;F)$ with the topology induced by the semi-norms
	\[
	\phi \mapsto \norm*{(1+\abs{\,\cdot\,})^k\norm{X\phi}_\rho}_\infty 
	\]
	as $k$ runs through $\N$, $X$ runs through the set of left-invariant differential operators on $E$, and $\rho$ runs through the set of continuous semi-norms on $F$.
	
	Now, let $C$ be a closed subset of $E$, and let $N_{E,C,F}$ be the set of $\phi \in \Sc(E;F)$ which vanish on $C$. Then, we shall define $\Sc_E(C;F)\coloneqq \quot{\Sc(E;F)}{N_{E,C,F}}$; we shall omit to denote $E$ when it is clear by the context. 
	We shall simply write $\Sc_E(C)$ instead of  $\Sc_E(C;\C)$.
\end{definition}

\begin{proposition}\label{prop:9}
	Let $F$ be a Fréchet space over $\C$, and $E$ a homogeneous group. Then, the bilinear mapping $\Sc(E)\times F\ni (\phi,v)\mapsto [h\mapsto \phi(h) v]\in \Sc(E;F)$ induces an isomorphism
	\[
	\Sc(E)\widehat \otimes F\to \Sc(E;F).
	\]
\end{proposition}

The proof is similar to that of~\cite[Theorem 51.6]{Treves} and is omitted.

\begin{proposition}\label{prop:12}
	Let $E_1,E_2$ be two homogeneous groups, and let $C_1,C_2$ be two closed subspaces of $E_1,E_2$, respectively. Then, $\Sc_{E_1\times E_2}(C_1\times C_2)$ is canonically isomorphic to $\Sc_{E_1}(C_1)\widehat \otimes \Sc_{E_2}(C_2)$.
\end{proposition}

\begin{proof}
	Define
	\[
	\Psi_{E,C}\colon \Sc(E)\ni \phi \mapsto (\phi(x))_{x\in C}\in \C^C
	\]
	for every homogeneous group $E$ and for every closed subspace $C$ of $E$. 
	Then, clearly $N_{E,C,\C}$ is the kernel of $\Psi_{E,C,\C}$. Now, observe that, with a slight abuse of notation, $\Psi_{E_1\times E_2, C_1\times C_2}=\Psi_{E_1,C_1}\widehat \otimes \Psi_{E_2,C_2}$ (cf.~\cite[Proposition 6 of Chapter I, §1, No.\ 3]{Grothendieck}). Therefore,~\cite[Proposition 3 of Chapter I, § 1, No.\ 2]{Grothendieck} implies that $N_{E_1\times E_2, C_1\times C_2, \C}$ is the closed vector subspace of $\Sc(E_1)\widehat\otimes \Sc(E_2)$ generated by the tensors of the form $\phi_1\otimes \phi_2$, with $\Psi_{E_1,C_1}(\phi_1)=0$ or $\Psi_{E_2,C_2}(\phi_2)=0$. By the same reference, we see that $N_{E_1\times E_2, C_1\times C_2, \C}$ is also the kernel of the canonical projection $\Sc(E_1)\widehat \otimes \Sc(E_2)\to \Sc_{E_1}(C_1)\widehat \otimes \Sc_{E_2}(C_2)$, so that the assertion follows.
\end{proof} 

\begin{lemma}\label{lem:3}
	Assume that $\Ac=\Set{A_1,A_2}$. Take $m\in L^\infty(\beta_{\Lc'_{ A'}})$ such that $\Kc_{\Lc'_{ A'}}\in \Sc(G_{ A'})$ and $T\in \Sc'(G_{A_2})$. Then, 
	\[
	\langle T, g_2\mapsto\Kc_{\Lc'_{ A'}}(m)(\,\cdot\,,g_2)\rangle\in \Sc_{\Lc_{A_1}}(G_{A_1}).
	\]
\end{lemma}

The proof is similar to that of Corollary~\ref{cor:2}, the only difference being that here one has to approximate $T$ in $\Sc'(G_{A_2})$ by a sequence of measures with compact support.

\begin{theorem}\label{teo:5}
	If $\Lc_A$ satisfies property $(S)$ for every $A\in \Ac$, then $\Lc'_{ A'}$ satisfies property $(S)$.
\end{theorem}

\begin{proof}
	{\bf1.} Proceeding by induction, we may reduce to the case in which $\Ac=\Set{A_1,A_2}$. In order to simplify the notation, we shall simply write $G_j$, and $\Sc(\sigma(\Lc_{A_j}))$ instead of $G_{A_j}$ and $\Sc_{E_{\Lc_{A_j}}}(\sigma(\Lc_{A_j}))$, respectively, for $j=1,2$.
	Now, take $m\in L^\infty(\beta_{\Lc'_{ A'}})$ such that $\Kc_{\Lc'_{ A'}}\in\Sc(G)$. Then, Corollary~\ref{cor:5}, Proposition~\ref{prop:6:2} and Fubini's theorem imply that
	\[
	\Mc_{\Lc_{A_1}}[g_1\mapsto\Mc_{\Lc_{A_2}}[\Kc_{\Lc'_{ A'}}(m)(g_1,\,\cdot\,)](\lambda_2)](\lambda_1)=m(\lambda_1, \lambda_2) 
	\]
	for $\beta_{\Lc_{A_1}}$-almost every $\lambda_1\in E_{\Lc_{A_1}}$ and for $\beta_{\Lc_{A_2}}$-almost every $\lambda_2\in E_{\Lc_{A_2}}$. 
	Observe that Lemma~\ref{lem:3} implies that $\Kc_{\Lc'_{ A'}}(m)(g_1,\,\cdot\,)\in \Sc_{\Lc_{A_2}}(G_{2})$ for every $g_1\in G_{1}$, and that by assumption $\Mc_{\Lc_{A_2}}$ induces a continuous linear mapping from $\Sc_{\Lc_{A_2}}(G_{2}) $ onto $\Sc(\sigma(\Lc_{A_2}))$. Therefore, the map $g_1 \mapsto \Mc_{\Lc_{A_2}}[\Kc_{\Lc'_{ A'}}(m)(g_1,\,\cdot\,)]$ defines an element of $\Sc(G_{1}; \Sc(\sigma(\Lc_{A_2})))$. 
	
	{\bf2.} Let us prove that the mapping $g_1 \mapsto \Mc_{\Lc_{A_2}}[\Kc_{\Lc'_{ A'}}(m)(g_1,\,\cdot\,)]$ belongs to the space $\Sc_{\Lc_{A_1}}(G_{1})\widehat \otimes \Sc(\sigma(\Lc_{A_2}))$. Take $T\in \Sc(\sigma(\Lc_{A_2}))'$; then, Lemma~\ref{lem:3} implies that
	\[
	[g_1 \mapsto (T\Mc_{\Lc_{A_2}})[\Kc_{\Lc'_{ A'}}(m)(g_1,\,\cdot\,)]] \in \Sc_{\Lc_{A_1}}(G_{1})
	\]
	since $T\Mc_{\Lc_{A_2}}$ defines an element of $\Sc_{\Lc_{A_2}}(G_{2})'$,
	which can be extended to an element of $\Sc'(G_2)$. 
	Next, observe that~\cite[Proposition 50.4]{Treves} implies that 
	\[
	\Sc_{\Lc_{A_1}}(G_{1})\widehat \otimes \Sc(\sigma(\Lc_{A_2}))\cong \Lc(\Sc(\sigma(\Lc_{A_2}))'; \Sc_{\Lc_{A_1}}(G_{1}))
	\]
	since $\Sc(\sigma(\Lc_{A_2}))$ is nuclear thanks to~\cite[Proposition 50.1]{Treves}. Now, the mapping 
	\[
	g_1 \mapsto \Mc_{\Lc_{A_2}}[\Kc_{\Lc'_{ A'}}(m)(g_1,\,\cdot\,)]
	\]
	belongs to $\Sc(G_1; \Sc(\sigma(\Lc_{A_2})))$; arguing as above, we see that this latter space is the canonical image of $\Sc(G_1)\widehat \otimes \Sc(\sigma(\Lc_{A_2}))\cong \Lc(\Sc(\sigma(\Lc_{A_2}))'; \Sc(G_{1}))$, so that the preceding arguments imply our claim.
	
	{\bf3.} Now, by assumption $\Mc_{\Lc_{A_1}}$ induces a continuous linear map from $\Sc_{\Lc_{A_1}}(G_{1})$ into $\Sc(\sigma(\Lc_{A_1}))$, so that we have the continuous linear mapping
	\[
	\Mc_{\Lc_{A_1}}\widehat \otimes I_{\Sc(\sigma(\Lc_{A_2}))  } \colon \Sc_{\Lc_{A_1}}(G_{1})\widehat \otimes \Sc(\sigma(\Lc_{A_2}))\to \Sc(\sigma(\Lc_{A_1}))\widehat \otimes \Sc(\sigma(\Lc_{A_2}));
	\] 
	in addition, for every $T\in \Sc(\sigma(\Lc_{A_2}))'$ and for every $\lambda_1\in \sigma(\Lc_{A_1})$, 
	\[
	\begin{split}
	&\langle T, \left(\Mc_{\Lc_{A_1}}\widehat \otimes I_{\Sc(\sigma(\Lc_{A_2}))}\right)( g_1 \mapsto \Mc_{\Lc_{A_2}}[\Kc_{\Lc'_{ A'}}(m)(g_1,\,\cdot\,)] )(\lambda_1)\rangle= \Mc_{\Lc_{A_1}}[g_1 \mapsto T\Mc_{\Lc_{A_2}}[\Kc_{\Lc'_{ A'}}(m)(g_1,\,\cdot\,)]](\lambda_1)
	\end{split}
	\]
	(reason as in~{\bf2}). Choosing $T=\delta_{\lambda_2}$ for $\lambda_2\in \sigma(\Lc_{A_2})$, and taking into account Proposition~\ref{prop:3}, we see that the mapping
	\[
	\sigma(\Lc'_{ A'})\ni(\lambda_1,\lambda_2)\mapsto \Mc_{\Lc_{A_1}}(g_1 \mapsto \Mc_{\Lc_{A_2}}[\Kc_{\Lc'_{ A'}}(m)(g_1,\,\cdot\,)](\lambda_2))(\lambda_1)
	\]
	extends to an element $m_0$ of $\Sc(E_{\Lc'_{ A'}})$. Now,~{\bf1} implies that $m_0(\lambda_1,\lambda_2)=m(\lambda_1,\lambda_2)$ for $\beta_{\Lc_{A_1}}$-almost every $\lambda_1\in E_{\Lc_{A_1}}$ and for $\beta_{\Lc_{A_2}}$-almost every $\lambda_2\in E_{\Lc_{A_2}}$. Since both $m$ and $m_0$ are $\beta_{\Lc'_{ A'}}$-measurable, Tonelli's theorem implies that $m=m_0$ $\beta_{\Lc'_{ A'}}$-almost everywhere.
	The assertion follows.
\end{proof}

\section{Image Families}\label{sec:8}

In this section we shall fix a Rockland family $\Lc_A$ on a homogeneous group $G$; we consider $\Lc_A$ as `known' and we study an `image family' $P(\Lc_A)$, where $P\colon \R^A\to \R^\Gamma$ is a polynomial mapping with homogeneous components, and $\Gamma$ is a finite set. We shall investigate what we can say about $P(\Lc_A)$ on the base of our knowledge of $\Lc_A$.

\begin{proposition}\label{prop:1:2}\label{prop:7:4}
	The following statements are equivalent:
	\begin{enumerate}
		\item $P(\Lc_A)$ is a Rockland family;
		
		\item the restriction of $P$ to $\sigma(\Lc_A)$ is proper, that is, $P(\lambda)\neq 0$ for every $\lambda\in \sigma(\Lc_A)$ such that $\abs{\lambda}=1$.
	\end{enumerate}
	
	In addition, if $P(\Lc_A)$ is a Rockland family, then:
	\begin{enumerate}
		\item[(i)] $\mi_{P(\Lc_A)}=P_*(\mi_{\Lc_A})$ and $\sigma(P(\Lc_A))=P(\sigma(\Lc_A))$;
		
		\item[(ii)] a $\beta_{P(\Lc_A)}$-measurable function $m\colon E_{P(\Lc_A)}\to \C$ admits a kernel if and only if $m\circ P$ admits a kernel; in this case,
		\[
		\Kc_{P(\Lc_A)}(m)= \Kc_{\Lc_A}(m\circ P);
		\]
		
		\item[(iii)] $\beta_{P(\Lc_A)}=P_*(\beta_{\Lc_A})$.
	\end{enumerate}
\end{proposition}

\begin{proof}
	By spectral theory,  $\mi_{P(\Lc_A)}=P_*(\mi_{\Lc_A})$ and  $\sigma(P(\Lc_A))=\overline{P(\sigma(\Lc_A))}$, without further assumptions on $P(\Lc_A)$.
	If $P(\Lc_A)$ is a Rockland family, then also~{\bf(ii)} holds by spectral theory again; as a consequence, also~{\bf(iii)} holds in this case.
	Then, we are reduced to proving the equivalence of~{\bf1} and~{\bf2}.
	
	{\bf 1 $\implies$ 2.} This follows from~\cite[Lemma 3.5.1]{Martini}.

	{\bf2 $\implies$ 1.} Notice first that the union of the families $\Lc_A$ and $P(\Lc_A)$ is clearly Rockland, so that the $P(\Lc_\alpha)$ are essentially self-adjoint on $C^\infty_c(G)$ with commuting closures.
	Let $S$ be the unit sphere of $E_{\Lc_A}$ corresponding to some homogeneous norm; take $\tau_1\in C^\infty(S)$ such that $\tau_1=1$ on a neighbourhood of $\sigma(\Lc_A)\cap S$ and such that $\tau_1$ is supported in $\Set{x\in S\colon 2\abs{P(x)}\Meg \min_{S\cap\sigma(\Lc_A)} \abs{P}}$. Then, extend $\tau_1$ to a homogeneous function of degree $0$.
	In addition, take $\tau_2\in C^\infty_c(E_{\Lc_A})$ so that $\tau_2=1$ on a neighbourhood of $0$.
	Now, if $m\in \Sc(E_{P(\Lc_A)})$, then clearly $[\tau_2+(1-\tau_2)\tau_1] (m\circ P)\in \Sc(E_{\Lc_A})$, so that $\Kc_{P(\Lc_A)}(m)\in \Sc(G)$. 
	The assertion follows. 	
\end{proof}

\begin{proposition}\label{prop:7:3}
	Assume that $P(\Lc_A)$ is a Rockland family, and take a disintegration $(\beta_{\lambda'})_{\lambda'\in E_{P(\Lc_A)}}$ of $\beta_{\Lc_A}$ relative to $P$. Then,
	\[
	\chi_{P(\Lc_A)}(\lambda',g)= \int_{E_{\Lc_A}} \chi_{\Lc_A}(\lambda,g)\,\dd \beta_{\lambda'}(\lambda)
	\]
	for $(\beta_{P(\Lc_A)}\otimes \nu_G)$-almost every $(\lambda',g)\in E_{P(\Lc_A)}\times G$.
\end{proposition}

Notice that the existence of a disintegration follows from~\cite[Theorem 1 of Chapter VI, § 3, No.\ 1]{BourbakiInt1}. 
Then, the proof amounts to showing that both sides of the asserted equality have the same integrals when multiplied by elements of $C^\infty_c(E_{\Lc_A})\otimes C^\infty_c(G)$; it is omitted.

\smallskip

Now, consider property $(RL)$. Assume that $\Lc_A$ satisfies property $(RL)$, and take $m\in L^\infty(\beta_{P(\Lc_A)})$ such that $\Kc_{P(\Lc_A)}(m)\in L^1(G)$. Then, there is $\widetilde m\in C_0(E_{\Lc_A})$ such that $m\circ P=\widetilde m$ $\beta_{\Lc_A}$-almost everywhere.
In Section~\ref{sec:6}, we shall study this situation in a general setting, seeking  conditions under which $\widetilde m$ is constant on the fibres of $P$ in $\sigma(\Lc_A)$. Since $P$ is proper, this implies that $\widetilde m$ is the composite of a continuous function with $P$, at least on $\sigma(\Lc_A)$. 
Notice, however, that sometimes it is more convenient to argue on suitable subsets of the spectrum.

Property $(S)$ is studied in a similar way, making use of the results of Sections~\ref{sec:6} and~\ref{sec:7}. 

\section{Composite Functions: Continuous Functions}\label{sec:6}

In this section we consider the following problem:
given three Polish spaces $X, Y, Z$, a measure $\mi$ on $X$, a $\mi$-measurable mapping $\pi\colon X\to Y$, and a function $m\colon Y\to Z$ such that $m\circ \pi$ equals $\mi$-almost everywhere a continuous function, does $m$ equal $\pi_*(\mi)$-almost everywhere a continuous function?

To this end, we introduce the following definition.

\begin{definition}
	Let $X$ be a Polish space, $Y$ a set, $\mi$ a positive Radon measure on $X$, and  $\pi$ a  mapping from $X$ into $Y$.
	We say that two points $x,x'$ of $\Supp{\mi}$ are $(\mi,\pi)$-connected if $\pi(x)=\pi(x')$ and there are $x=x_1,\dots, x_k=x'\in \pi^{-1}(\pi(x))\cap \Supp{\mi}$ such that, for every $j=1,\dots, k$, for every neighbourhood $U_j$ of $x_j$ in $\Supp{\mi}$, and for every neighbourhood $U_{j+1}$ of $x_{j+1}$ in $\Supp{\mi}$, the set $ \pi^{-1}(\pi(U_j)\cap \pi(U_{j+1})) $ is not $\mi$-negligible.
	We say that $\mi$ is $\pi$-connected if every pair of elements of $ \Supp{\mi}$ having the same image under $\pi$ are $(\mi,\pi)$-connected.
\end{definition}

Observe that  $(\mi,\pi)$-connectedness actually depends only on the equivalence class of $\mi$ and the equivalence relation induced by $\pi$ on $X$.
In addition, notice that, if $Y$ is a topological space and $\pi$ is open at some point of each fibre (in the support of $\mi$), then $\mi$ is $\pi$-connected.

We emphasize that, in the definition of $(\mi,\pi)$-connectedness, the points $x_1,\dots,x_k$ are fixed \emph{before} considering their neighbourhoods. In other words,  if for every neighbourhood $U$ of $x$ in $\Supp{\mi}$ and for every  neighbourhood $U'$ of $x'$ in $\Supp{\mi}$ we found $x=x_1,\dots, x_k=x'$ and neighbourhoods $U_j$ of $x_j$ in $\Supp{\mi}$ so that $U=U_1$, $U'=U_k$ and, for every $j=1,\dots, k$, the set $\pi^{-1}(\pi(U_j)\cap \pi(U_{j+1})) $ were not $\mi$-negligible, then we would \emph{not} be able to conclude that $x$ and $x'$ are $(\mi,\pi)$-connected.

Now we can prove our main result. Notice that, even though its hypotheses are quite restrictive, it still gives rise to important consequences.

\begin{proposition}\label{prop:A:7}
	Let $X,Y,Z$ be three Polish spaces, $\pi\colon X\to Y$ a $\mi$-measurable mapping, and $\mi$ a  $\pi$-connected positive Radon measure on $X$.  
	Assume that $\pi$ is $\mi$-proper and that there is a disintegration $(\lambda_y)_{y\in Y}$ of $\mi$ relative to $\pi$ such that $\Supp{\lambda_y}\supseteq\Supp{\mi} \cap \pi^{-1}(y)$ for $\pi_*(\mi)$-almost every $y\in Y$.
	
	Take a continuous map $m_0\colon X\to Z$ such that there is map $m_1\colon Y\to Z$ such that $m_0(x)= (m_1\circ \pi)(x)$ for $\mi$-almost every $x\in X$. 
	Then, there is a $\pi_*(\mi)$-measurable mapping $m_2\colon Y\to Z$ such that $m_0=m_2\circ \pi$ \emph{pointwise} on $\Supp{\mi}$.
\end{proposition}

If $\pi$ is also proper, then $m_2$ is actually continuous on $\pi(\Supp{\mi})$. 

\begin{proof}
	Observe first that there is a  $\pi_*(\mi)$-negligible subset $N$ of $Y$ such that $m_1\circ \pi=m_0$ $\lambda_y$-almost everywhere for every $y\in Y\setminus N$. 
	Notice that we may assume that $\Supp{\mi}=X$ and that, if $y\in Y\setminus N$, then the support of $\lambda_y$ contains $\pi^{-1}(y)$.	
	Since $m_0$ is continuous and since $m_1\circ \pi$ is constant on the support of $\lambda_y$, it follows that $m_0$ is constant on $\pi^{-1}(y)$ for every $y\in Y\setminus N$.   
	
	Now, take $y\in \pi(X)\cap N$ and $x_1,x_2\in \pi^{-1}(y)$. Let $\Uf(x_1)$ and $\Uf(x_2)$ be the filters of neighbourhoods of $x_1$ and $x_2$, respectively.
	Assume first that $\pi(U_1)\cap \pi(U_2)  $ is not $\pi_*(\mi)$-negligible for every $U_1\in \Uf(x_1)$ and for every $U_2\in \Uf(x_2)$. 
	Take $U_1\in \Uf(x_1)$ and $U_2\in \Uf(x_2)$. 
	Then, there is $y_{U_1,U_2}\in \pi(U_1)\cap \pi(U_2)\setminus N$, and then $x_{h,U_1,U_2}\in U_h\cap \pi^{-1}(y_{U_1,U_2})$ for $h=1,2$. 
	Now, $m_0(x_{1,U_1,U_2})=m_0(x_{2,U_1,U_2})$ for every $U_1\in \Uf(x_1)$ and for every $U_2\in \Uf(x_2)$. 
	In addition, $x_{h,U_1,U_2}\to x_h$ in $X$ along the product filter of  $\Uf(x_1)$ and $ \Uf(x_2)$. 
	Since $m_0$ is continuous,  passing to the limit we see that $m_0(x_1)=m_0(x_2)$.
	Since $\mi$ is $\pi$-connected, this implies that $m_0$ is constant on $ P^{-1}(y)$ for \emph{every} $y\in \pi(X)$.
	The assertion follows.
\end{proof}

In the following proposition we give sufficient conditions in order that a measure be connected.

\begin{proposition}\label{prop:A:8}
	Let $E_1,E_2$ be two finite-dimensional vector spaces, $L\colon E_1\to E_2$ a linear mapping, $C$ a closed convex subset of $E_1$ and $\mi$ a positive Radon measure on $E_1$ with support $C$. Take a Polish subspace $X$ of $E_1$ so that $\mi(E_1\setminus X)=0$.
	Then, $\mi_X$ is $\restr{L}{X}$-connected.
\end{proposition}

Actually, there is no need that $X$ be a Polish space, but we did not consider Radon measures on more general Hausdorff spaces.

\begin{proof}
	We may assume that $C$ has non-empty interior.
	Then,  we may find a bounded convex open subset $U$ of $ C$ and an convex open neighbourhood $V$ of $0$ in $\ker L$ such that $U+V\subseteq C$.
	Take $r\in ]0,1]$ and $x,y\in C\cap X$ such that $y-x\in V$; take $R_x>0$ so that $U\subseteq B(x,R_x)$. 
	Then, for every $u\in U$ we have $y+ r(u-x)\in B(y,r R_x)\cap [y,y-x+u]\subseteq B(y,r R_x)\cap C$; analogously, $x+r(U-x)\subseteq B(x,r R_x)\cap C$.
	Since $L(x)=L(y)$, we infer that 
	\[
	L^{-1}(L( B(x, r R_x)\cap C \cap X)\cap L(B(y,r R_x)\cap C\cap X))\supseteq [x+r(U-x)]\cap X.
	\]
	Now, $x+r(U-x)$ is a non-empty open subset of $C=\Supp{\mi}$, so that $\mi_X([x+r(U-x)]\cap X)=\mi(x+r(U-x))>0$. 
	The arbitrariness of $r$ then implies that $x$ and $y$ are $(\mi,L)$-connected.
	The assertion follows easily. 
\end{proof}

Now we present a result on the disintegration of Hausdorff measures, which is particularly useful to check the assumptions of Proposition~\ref{prop:A:7}.
It is a corollary of~\cite[Theorem 3.2.22]{Federer}; we omit the proof and refer to~\cite{Federer} for any unexplained notation.

\begin{proposition}\label{prop:A:6}
	Let, for $j=1,2$, $E_j$ be an $\Hc^{k_j}$-measurable and countably $\Hc^{k_j}$-rectifiable subset of $\R^{n_j}$. 
	Assume that $k_2\meg k_1$, and let $P$ be a locally Lipschitz mapping of $E_1$ into $E_2$. 
	Take a positive function $f\in L^1_\loc(\chi_{E_1}\cdot \Hc^{k_1})$ and assume that $f(x)\,\ap J_{k_2} P(x)\neq 0$ for $\Hc^{k_1}$-almost every $x\in E_1$, and that $P$ is $(f\cdot \Hc^{k_1})$-proper.
	
	Then, the following hold:
	\begin{enumerate}
		\item  the mapping
		\[
		g\colon\R^{n_2}\ni y \mapsto \int_{P^{-1}(y)} \frac{f}{\ap J_{k_2} P}\,\dd \Hc^{k_1-k_2}
		\]
		is well-defined $\Hc^{k_2}$-almost everywhere and measurable; in addition,
		\[
		P_*(f\cdot \Hc^{k_1})= g\cdot \Hc^{k_2};
		\]
		
		\item the measure
		\[
		\beta_y\coloneqq \frac{1}{g(y)}\frac{f}{\ap J_{k_2} P} \chi_{P^{-1}(y)}\cdot \Hc^{k_1-k_2}
		\]
		is well-defined and Radon for $P_*(f\cdot \Hc^{k_1})$-almost every $y\in \R^{n_2}$; in addition, $(\beta_y)$ is a disintegration of $f\cdot \Hc^{k_1}$ relative to $P$;
		
		\item $\beta_y$ is equivalent to $\chi_{P^{-1}(y)}\cdot \Hc^{k_1-k_2} $ for $P_*(f\cdot \Hc^{k_1})$-almost every $y\in E_2$. 
	\end{enumerate}
\end{proposition}

Notice that, if $E_1$ is a submanifold of $\R^{n_1}$ and $P$ is of class $C^1$, then $\ap J_{k_2}P(x)$ is simply $\norm{\bigwedge^{k_2} T_x (P)}$ for every $x\in E_1$.

\section{Composite Functions: Schwartz Functions}\label{sec:7}

In this section we shall extend some results on composite differentiable functions by E.\ Bierstone, P.\ Milman and G.\ W.\ Schwarz to the case of Schwartz functions by means of techniques developed by F.\ Astengo, B.\ Di Blasio and F.\ Ricci.

We shall take advantage of the remarkable works of E.\ Bierstone, P.\ Milman and G.\ W.\ Schwarz about the composition of smooth functions on analytic manifolds, and we shall refer to~\cite{BierstoneMilman,BierstoneMilman2,BierstoneSchwarz} for any unexplained definition, in particular for the notion of (Nash) subanalytic sets.
As a matter of fact, in the applications we shall only need to know that any convex subanalytic set is automatically Nash subanalytic, since it is contained in an affine space of the same dimension, and that semianalytic sets are Nash subanalytic (cf.~\cite[Proposition 2.3]{BierstoneMilman}).

Our starting point is the following result (cf.~\cite[Theorem 0.2]{BierstoneMilman} and~\cite[Theorem 0.2.1]{BierstoneSchwarz}).
Here, $\Ec(\R^m)$ denotes the set of $C^\infty$ functions on $\R^m$ endowed with the topology of locally uniform convergence of all derivatives; $\Ec_{\R^n}(C)$ is the quotient of $\Ec(\R^n)$ by the space of $C^\infty$ functions which vanish on $C$.

\begin{theorem}\label{teo:8}
	Let $C$ be a closed subanalytic subset of $\R^n$ and let $P\colon \R^n\to \R^m$ be an analytic mapping. Assume that $P$ is proper on $C$ and that $P(C)$ is Nash subanalytic.
	Then, the canonical mapping
	\[
	\Phi\colon\Ec(\R^m)\ni \phi\mapsto \phi\circ P\in \Ec_{\R^n}(C)
	\]
	has a closed range, and admits a continuous linear section defined on $\Phi(\Ec(\R^n))$.
	
	In addition, $\psi\in \Ec_{\R^n}(C)$ belongs to the image of $\Phi$ if and only if for every $y\in \R^m$ there is $\phi_y\in \Ec(\R^m)$ such that, for every $x\in C$ such that $P(x)=y$, the Taylor series of $\phi_y\circ P$ and $\psi$ at $x$ differ by the Taylor series of a function of class $C^\infty$ which vanishes on $C$.
\end{theorem}

In order to simplify the notation, we shall simply say that $\psi $ is a formal composite of $P$ if the second condition of the statement holds.

Now, we are particularly interested in the case of Schwartz functions. 
The strategy developed in~\cite{AstengoDiBlasioRicci2} is the following: one first decomposes dyadically a given Schwartz function in the sum of dilates of a family of test functions with a suitable decay; then, one applies the section given by Theorem~\ref{teo:8}, truncates the resulting functions (so that they are still test functions), and finally sums their dilates.
In order to do that, however, one needs homogeneity.

\begin{theorem}\label{teo:7}
	Let $P\colon \R^n\to \R^m$ be a polynomial mapping, and assume that $\R^n$ and $\R^m$ are endowed with dilations such that $P(r\cdot x)=r\cdot P(x)$ for every $r>0$ and for every $x\in \R^n$.
	Let $C$ be a dilation-invariant subanalytic closed subset of $\R^n$, and assume that $P$ is proper on $C$ and that $P(C)$ is Nash subanalytic.
	Then, the canonical mapping
	\[
	\Phi\colon\Sc(\R^m)\ni \phi\mapsto \phi\circ P\in \Sc_{\R^n}(C)
	\]
	has a closed range and admits a continuous linear section defined on $\Phi(\Sc(\R^m))$.
	In addition, $\psi\in \Sc_{\R^n}(C)$ belongs to the image of $\Phi$ if and only if it is a formal composite of $P$.
\end{theorem}

As a matter of fact, in our applications $\R^n$ will be $E_{\Lc_A}$, while $\R^m$ will be $E_{P(\Lc_A)}$; $C$ will be (a subset of) $\sigma(\Lc_A)$. 
Then, Theorem~\ref{teo:7} gives some sufficient conditions in order that some $f\in \Sc_{P(\Lc_A)}(G)$, which has a Schwartz multiplier on $E_{\Lc_A}$, should have a Schwartz multiplier on $E_{P(\Lc_A)}$ (cf.~Section~\ref{sec:8}).

Notice, however, that sometimes it is convenient to take $C$ so as to be a portion of $\sigma(\Lc_A)$ such that $P(C)=\sigma(P(\Lc_A))$, since $\sigma(\Lc_A)$ need \emph{not} be subanalytic.

\begin{proof}
	For the first assertion, simply argue as in the proof of~\cite[Theorem 6.1]{AstengoDiBlasioRicci2} replacing the linear section provided by Schwarz and Mather with that of Theorem~\ref{teo:8}.
	
	As for the second part of the statement, notice first that it follows easily from Theorem~\ref{teo:8} when $\psi$ is compactly supported; since the image of $\Phi$ is closed, it follows by approximation in the general case.
\end{proof}

In the following result we give a simple but very useful application of  Theorem~\ref{teo:7}.

\begin{corollary}\label{cor:A:7}
	Let $V$ and $W$ be two finite-dimensional vector spaces,  $C$ a subanalytic closed convex cone in $V$, and $L$ a linear mapping of $V$ into  $W$ which is proper on $C$. 
	Take $m_1\in \Sc(V)$ and assume that there is $m_2\colon W\to \C$ such that $m_1=m_2\circ L$ on $C$. 
	Then, there is $m_3\in \Sc(W)$ such that $m_1=m_3\circ L$ on $C$.
\end{corollary}

\begin{proof}
	Observe first that we may assume $C$ has vertex $0$ and has non-empty interior.
	Observe, by the way, that $L(C)$ is subanalytic (cf.~\cite[Theorem 0.1 and Proposition 3.13]{BierstoneMilman2}), hence Nash subanalytic.
	Now, fix $x\in C$. 
	Since the interior of $C$ is non-empty, it is clear that $C$ is a total subset of $V$, so that we may find a free family $(v_j)_{j\in J}$ in $C$ which generates an algebraic complement $V'$ of $\ker L$ in $V$. 
	In addition, since either $x=0$ or $x\not \in \ker L$, we may assume that $x\in V'$.  
	Let $L'\colon W\to V$ be the composite of the inverse of the restriction of $L$ to $V'$ with the natural immersion of $V'$ in $V$. 
	Then, $L'$ is a linear section of $L$. 
	
	Define $m'\coloneqq m_1\circ L'$, so that $m'\in \Ec(W)$. 
	Next, define $C'\coloneqq V'\cap C$, so that $C'$ is a closed convex cone with non-empty interior in $V'$, since it contains the non-empty open set $\sum_{j\in J} \R_+^* v_j$. 
	Take $z\in C'$ and any $y\in C\cap[x+\ker L]$. 
	Then,  $x+z=(L'\circ L) (x+z)=(L' \circ L )(y+z)$, so that $m_1=m'\circ L$ on $y+C'$. 
	Since $m_1$ is constant on the intersections of $C$ with the translates of $\ker L$, the same holds on $  C\cap (y+C'+\ker L)$. Now, denote by $\open{C'}$ the interior of $C'$ in $V'$. 
	Then, $y+\open{C'}+\ker L$ is an open convex set and $y$ is adherent to $C\cap (y+\open{C'}+\ker L)$, so that the Taylor polynomials of every fixed order of $m_1$ and $m'\circ L$ about $y$ coincide on $C\cap (y+\open{C'}+\ker L)$, hence on $V$. 
	Since this holds for every $y\in C\cap [x+\ker L]$,  Theorem~\ref{teo:7} implies that there is $m_3\in \Sc(W)$ such that $m_1=m_3\circ L$ on $C$.
\end{proof}

\section{Quadratic Operators on $2$-Step Stratified Groups}

A connected Lie group $G$ is called $2$-step  nilpotent if $[\gf,[\gf,\gf]]=0$, where $\gf$ is the Lie algebra of $G$.
The group $G$ is $2$-step stratified if, in addition, it is simply connected and  $\gf=\gf_1\oplus \gf_2$, with $[\gf_1,\gf_1]=[\gf,\gf]=\gf_2$.

Notice that, if $G$ is a simply connected $2$-step nilpotent group, then it is `stratifiable,' that is, for every algebraic complement $\gf_1$ of $\gf_2\coloneqq [\gf,\gf]$, the decomposition $\gf=\gf_1 \oplus \gf_2$ turns $G$ into a stratified group. Nevertheless, $G$ may be endowed with many different structures of a stratified group; when we speak of a $2$-step stratified group, we then mean that an algebraic complement of $[\gf,\gf]$ is fixed. 

A $2$-step stratified group is endowed with the canonical dilations, that is $r\cdot (X+Y)=r X+ r^2 Y$
for every $r>0$, for every $X\in \gf_1$ and for every $Y\in \gf_2$. Thus, $G$ becomes a homogeneous group. 

\begin{definition}
	Let $G$ be a $2$-step stratified group. Then, for every $\omega\in \gf_2^*$ we shall define
	\[
	B_\omega\colon \gf_1\times \gf_1\ni(X,Y)\mapsto \langle \omega, [X,Y]\rangle.
	\]
	Then,  $G$ is an $MW^+$ group if $B_\omega$ is non-degenerate for some $\omega\in \gf_2^*$ (cf.~\cite{MooreWolf} and also~\cite{MullerRicci}).
	A Heisenberg group is an $MW^+$ group with one-dimensional centre.
\end{definition}

\begin{definition}
	Take $d\in \N^*$, and let $\gf$ be the free Lie algebra on $d$ generators. Then, the quotient $\gf'$ of $\gf$ by its ideal $[\gf, [\gf,\gf]]$ is the free $2$-step nilpotent Lie algebra on $d$ generators.
	The simply connected Lie group with Lie algebra $\gf'$ is called the free $2$-step nilpotent Lie group on $d$ generators.
\end{definition}

Now,  to every symmetric bilinear form $Q$ on $\gf_1^*$ we can associate a differential operator on $G$ as follows:
\[
\Lc\coloneqq -\sum_{\ell,\ell'} Q(X_\ell^*, X_{\ell'}^*) X_{\ell} X_{\ell'},
\] 
where $(X_\ell)$ is a basis of $\gf_1$ with dual basis $(X^*_\ell)$. As the reader my verify, $\Lc$ does not depend on the choice of $(X_\ell)$; actually, one may prove that $-\Lc$ is the symmetrization of the quadratic form induced by $Q$ on $\gf^*$ (cf.~\cite[Theorem 4.3]{Helgason}).

\begin{lemma}\label{lem:10:10}
	Let $Q$ be a symmetric bilinear form on $\gf_1^*$, and let $\Lc$ be the associated operator. 
	Then, $\Lc$ is formally self-adjoint if and only if $Q$ is real. 
	In addition, $\Lc$ is formally self-adjoint and hypoelliptic if and only if $Q$ is non-degenerate and either positive or negative.
\end{lemma}

\begin{proof}
	The first assertion follows from the fact that the formal adjoint of $\Lc$ is associated with $\overline Q$. 
	The last assertion then follows from~\cite{Hormander3}.
\end{proof}

Next, we  show how to put $\Lc$ in a particularly convenient form according to the chosen $\omega\in \gf_2^*$.

\begin{definition}
	Let $V$ be a vector space and $\Phi$ a bilinear form on $V$. Then, we shall define
	\[
	\dd_\Phi\colon V\ni v \mapsto \Phi(\,\cdot\,,v)\in V^*.
	\]
\end{definition}

Notice that any algebraic complement of the radical of a skew-symmetric bilinear form on a finite-dimensional vector space is symplectic. Therefore, by~\cite[Corollary 5.6.3]{AbrahamMarsden} we deduce the following result.

\begin{proposition}\label{prop:10:8}
	Let $V$ be a finite-dimensional vector space over $\R$, let $\sigma$ be a skew-symmetric bilinear form on $V$, and let $Q$ be a positive, non-degenerate bilinear form on $V$.
	Then, there are a basis $(v_j)_{j=1,\dots,m}$ of $V$ and a positive integer $n\meg \frac{m}{2}$ such that the following hold:
	\begin{itemize}
		\item $Q(v_j,v_j)=Q(v_{n+j},v_{n+j})>0$ for every $j=1,\dots, n$;
		
		\item $Q(v_j,v_k)=0$ for every $j,k\in \Set{1,\dots,m}$ such that $j\neq k$ and either $j\meg 2 n$ or $k\meg 2 n$;
		
		\item for every $j,k=1,\dots, m$,
		\[
		\sigma(v_j,v_k)=\begin{cases}
		1 & \text{if $j\in \Set{1,\dots,n}$ and $k=n+j$};\\
		-1 & \text{if $j\in \Set{n+1,\dots,2 n}$ and $k=j-n$};\\
		0 & \text{otherwise}.
		\end{cases}
		\]
	\end{itemize} 
\end{proposition}

Observe that $Q(v_j,v_j)$ is the eigenvalue of $\abs{\dd_{Q}^{-1}\circ \dd_\sigma }$ corresponding to $v_j$, where the absolute value is computed with respect to $Q$.

\section{Plancherel Measure and Integral Kernel}\label{sec:11}

In this section, $G$ denotes a $2$-step stratified group of dimension $n$ which does \emph{not} satisfy the $MW^+$ condition, $Q$ a  symmetric bilinear form on $\gf_1^*$, and $(T_1,\dots, T_{n_2})$ a basis of $\gf_2$.
We shall denote by $\Lc $ the sub-Laplacian induced by $Q $ and we shall assume that $\Lc_A\coloneqq (\Lc, (-i T_k)_{k=1,\dots,n_2})$ is a Rockland family, that is, that $\Lc$ is a hypoelliptic sub-Laplacian, up to a sign. 
Indeed, if $\pi_0$ is the projection of $G$ onto its abelianization, then $\dd \pi_0(\Lc_A)$ is a Rockland family, so that $\Fc(\dd \pi_0(\Lc_A))$ vanishes only at $0$. 
Since $\dd \pi_0(T_k)=0$  for every $k=1,\dots, n_2$, this implies that $Q$ is non-degenerate and either positive or negative; hence, $\Lc$ is a hypoelliptic sub-Laplacian, up to a sign.
We may then assume that $Q$ is positive and non-degenerate.

We shall also endow $\gf$ with a scalar product for which  $\gf_1$ and $\gf_2$ are orthogonal,
and which induces $\widehat Q$ on $\gf_1$.
Then, we may endow $\gf$ with the translation-invariant measure $\Hc^n$;
up to a normalization, \emph{we may then assume that $(\exp_G)_*(\Hc^n)$ is the chosen Haar measure on $G$.}
We  shall endow $\gf_2^*$ with the scalar product induced by that of $\gf_2$, and then with the corresponding Lebesgue measure. 

Define
\[
J_{Q, \omega}\coloneqq \dd_{Q}\circ \dd_{B_\omega}\colon \gf_1 \to \gf_1
\]
for every $\omega\in \gf_2^*$, and define $d\coloneqq \min_{\omega\in \gf_2^*}\dim \ker \dd_{B_\omega}$, so that $d>0$ since $G$ is not an $MW^+$ group.
We denote by $W$ the set of $\omega\in \gf_2^*$ such that $\dim \ker \dd_{B_\omega}>d$. Define $n_1\coloneqq \frac{1}{2}(\dim \gf_1-d)$, and observe that $n_1=0$ if and only if $G$ is abelian.

We denote by $\Omega$ the set of $\omega\in \gf_2^*\setminus W$ where $\card\left(\sigma( \abs{J_{Q,\omega} } )\setminus\Set{0}\right)$ attains its maximum $\overline h$. 
As the next lemma shows, if $G$ is not abelian, then  $\Omega$ is open and dense.

\begin{lemma}\label{lem:2}
	The sets $W$ and $\gf_2^*\setminus \Omega$ are algebraic varieties.
\end{lemma}

As the proof shows, the multiplicities of the eigenvalues are constant on $\Omega$.

\begin{proof}
	Define $P_{\omega}$ so that $X^d P_{\omega}(X)$ is the characteristic polynomial of $-J_{Q,\omega}^2$. 
	Then, it is clear that $W$ is the zero locus of the polynomial mapping $\omega \mapsto P_{\omega}(0)$, so that it is an algebraic variety.
	
	Next,  take $k\in \Set{1,\dots, n_1}$ and let $\Pf_k$ be the set of partitions of $\Set{1,\dots, n_1}$ into $k$ non-empty sets. Define
	\[
	P_k(X_{1},\dots, X_{n_1})\coloneqq \prod_{\Kc\in \Pf_k} \sum_{K\in \Kc} \sum_{k_1,k_2\in K} (X_{k_1}-X_{k_2})^2,
	\]
	so that $P_k$ is a $\Sf_{n_1}$-invariant polynomial. 
	Take $\widetilde \mi_{1,\omega},\dots, \widetilde \mi_{n_1,\omega}\Meg 0$ so that the eigenvalues of $J_{Q,\omega}$  are $0,\dots, 0, \pm i \widetilde \mi_{1,\omega},\dots, \pm i \widetilde \mi_{n_1,\omega}$ for every $\omega\in \gf_2^*$.
	Now, the mapping $\omega \mapsto P_k(\widetilde \mi_{1,\omega}^2,\dots, \widetilde \mi_{n_1,\omega}^2)$ is a $\Sf_{n_1}$-invariant polynomial mapping in the roots of the polynomial $P_{\omega}$; hence, it is a polynomial mapping (cf.~\cite[Theorem 1 of Chapter IV, § 6, No.\ 1]{BourbakiA2}). 
	Therefore, the set of $\omega\in \gf_2^*$ such that $P_k(\widetilde \mi_{1,\omega},\dots, \widetilde \mi_{n_1,\omega})=0$ is an algebraic variety $W_k$. 
	In addition, it is clear that $\Omega$ is the complement of $W\cup W_{\overline h-1}$, so that it is open in the Zariski topology.
\end{proof}

\begin{proposition}\label{prop:11:5}
	There are four analytic mappings
	\begin{align*}
		&\mi\colon \Omega \to (\R_+^*)^{\overline{h}} 
		& &P\colon \Omega \to \Lc(\gf_1)^{\overline{h}} 
		& &P_0\colon \gf_2^*\setminus W\to \Lc(\gf_1)
		& &\rho\colon \Omega \to \Set{1,\dots, \overline h}^{n_1}
	\end{align*}
	such that the following hold:
	\begin{itemize}
		\item the mapping
		\[
		\Omega  \ni \omega \mapsto \mi_{\rho_{k,\omega},\omega}\in \R_+
		\]
		extends to a continuous mapping $\omega \mapsto \widetilde \mi_{k,\omega}$ on $\gf_2^*$ for every $k=1,\dots, n_1$;
		
		\item for every $h=0,\dots, \overline h$ and for every $\omega\in \Omega$ (for every $\omega\in \gf_2^*\setminus W$, if $h=0$), $P_{h,\omega}$ is a $B_\omega$- and $\widehat Q$-self-adjoint projector of $\gf_1$;
		
		\item if $h=1,\dots, \overline h$ and $\omega\in \Omega$, then $\tr P_{h,\omega}=2 \card( \Set{k\in \Set{1,\dots, n_1}\colon \rho_{k,\omega}=h} )$;
		
		\item $\sum_{h=0}^{\overline h} P_{h,\omega}=I_{\gf_1}$ and $\sum_{h=1}^{\overline h} \mi_{h,\omega} P_{h,\omega}= \abs{J_{Q, \omega}}$  for every $\omega\in \Omega$;
		
		\item $P_{0,\omega}(\gf_1)=\ker \dd_{B_\omega}$ for every $\omega\in \gf_2^*\setminus W$.
	\end{itemize}
\end{proposition}

The proof is omitted, since it basically consists of straightforward generalizations of the arguments of~\cite[§ 1.3--4 and § 5.1 of Chapter II]{Kato}.

\begin{definition}
	We define  $\mi, \widetilde \mi, P$ and  $P_0$ as in Proposition~\ref{prop:11:5}. 
	In addition, we define
	$\vect{n_{1}}\colon \Omega \to (\N^*)^{\overline{h}} $ so that $n_{1,h,\omega}=\frac{1}{2}\tr P_{h,\omega}$ for every $h=1,\dots, \overline{h} $ and for every $\omega\in \Omega$.
	
	Furthermore, we shall sometimes identify $\mi_\omega$ with the linear mapping
	\[
	\R^{\overline{h}} \ni \lambda \mapsto \sum_{h=1}^{\overline{h} }\mi_{h,\omega} \lambda_h\in \R
	\]
	for every $\omega\in \Omega$. Analogous notation for $\widetilde \mi_\omega$.
\end{definition}

With the above notation, we have $\mi_{\omega}(\vect{n_{1,\omega}})=\sum_{h=1}^{\overline{h} }\mi_{h,\omega} n_{1,h,\omega}$.
Observe, in addition, that the index $1$ in $\vect{n_1}$ refers to the first layer $\gf_1$, just as the index $2$ in $n_2$ refers to the second layer $\gf_2$.

\begin{corollary}\label{cor:11:2}
	The function $\omega \mapsto \mi_{\omega}(\vect{n_{1,\omega}})=\widetilde \mi_{\omega}(\vect{1}_{n_1})$ is a norm on $\gf_2^*$ which is analytic on $\gf_2^*\setminus W$. 
\end{corollary}

\begin{proof}
	Observe that
	\[
	2\mi_{\omega}(\vect{n_{1,\omega}})= \norm{ J_{Q , \omega} }_1=\norm{J_{Q,\omega}+P_{0,\omega}}_1-d
	\]
	for every $\omega\in \gf_2^*$, and that the linear mapping $\omega \mapsto J_{Q, \omega}$ is one-to-one since $G$ is stratified. The assertion follows.
\end{proof}

\begin{definition}
	By an abuse of notation, we shall denote by $(x,t)$ the elements of $G$, where $x\in \gf_1$ and $t\in \gf_2$, thus identifying $(x,t)$ with $\exp_G(x,t)$. 
	For every $x\in \gf_1$ and for every $\omega\in \gf_2^*\setminus W$, we shall define
	\[
	x_{0,\omega}\coloneqq P_{0,\omega} (x),
	\]
	while, for every $\omega\in \Omega$ and for every $h=1,\dots, \overline{h} $,
	\[
	x_{h,\omega}\coloneqq \sqrt{\mi_{h,\omega}}P_{h,\omega} (x).
	\] 
	By an abuse of notation, we shall write $x_\omega$ instead of $\sum_{h=1}^{\overline h}x_{h,\omega}$, so that $\abs{x_{\omega}}=\left( \sum_{h=1}^{\overline{h} } \abs{x_{h,\omega}}^2\right)^{\sfrac{1}{2}}$.
\end{definition}

\begin{proposition}\label{prop:11:1}
	The mapping
	\[
	\gf_1 \times \Omega  \ni (x,\omega)\mapsto \sum_{h=1}^{\overline{h} } x_{h,\omega}
	\]
	extends uniquely to a continuous function on $\gf_1\times\gf_2^*$ which is analytic on $\gf_1 \times (\gf_2^*\setminus W)$.
\end{proposition}	

\begin{proof}
	Observe that, for every $\omega\in \gf_2^*$, $-J_{Q,\omega}^2=J_{Q,\omega}^*J_{Q,\omega}$ is positive, and that
	\[
	-J_{Q,\omega}^2+P_{0,\omega}
	\]
	is positive and non-degenerate as long as $\omega\not \in W$.
	Therefore, the mapping
	\[
	\omega\mapsto \sqrt[4]{-J_{Q,\omega}^2}=\sqrt[4]{-J_{Q,\omega}^2+P_{0,\omega}}- P_{0,\omega}\in \Lc(\gf_1)
	\]
	is continuous on $\gf_2^*$ and analytic on $\gf_2^*\setminus W$ thanks to~\cite[Proposition 10 of Chapter I, § 4, No.\ 8]{BourbakiTS}.\footnote{For what concerns continuity, just observe that $\sqrt[4]{\,\cdot\,}$ is continuous on the cone of positive endomorphisms of $\gf_1$, which is the closure of the cone of non-degenerate positive endomorphisms of $\gf_1$, as in~\cite[p.\ 85]{Hormander2}.}
	Then, it suffices to observe that
	\[
	\sqrt[4]{-J_{Q,\omega}^2}(x)= \sum_{h=1}^{\overline{h}} x_{h,\omega}
	\]
	for every $\omega\in \Omega$ and for every $x\in \gf_1$.
\end{proof}

\begin{definition}
	Define $G_\omega$, for every $\omega\in \gf_2^*$,  as the quotient of $G$ by its normal subgroup $\exp_G(\ker \omega)$. 
	
	Then, $G_0$ is the abelianization of $G$, and we identify it with $\gf_1$.
	If $\omega\neq 0$, then we shall identify $G_\omega$ with $\gf_1 \oplus \R$, endowed with the product
	\[
	(x_1,t_1) (x_2,t_2)\coloneqq\left(x_1+x_2, t_1+t_2+\frac{1}{2}B_\omega(x_1,x_2) \right)
	\]
	for every $x_1,x_2\in \gf_1$ and for every $t_1,t_2\in \R$. Hence, 
	\[
	\pi_\omega(x,t)=(x,\omega(t))
	\]
	for every $(x,t)\in G$.
\end{definition}

\begin{definition}
	For every $\omega\in \gf_2^*\setminus W$, define $\abs{\Pfaff(\omega)}\coloneqq \prod_{h=1}^{\overline{h}} \mi_{h,\omega}^{n_{1,h,\omega}}$, the Pfaffian of $\omega$ (cf.~\cite{AstengoCowlingDiBlasioSundari}). 
\end{definition}

Now we are in position to find the Plancherel measure and the integral kernel associated with $\Lc_A$. This is done by means of the explicit knowledge of the Plancherel and inversion formulae of $G$ (cf.~\cite{AstengoCowlingDiBlasioSundari}) and the following weak version of Poisson's formula (cf.~\cite[Proposition 5.4]{MartiniRicciTolomeo} for a proof in a slightly different setting).

\begin{proposition}\label{prop:2:9}
	Let $\Lc'_{A'}$ be a Rockland family on a homogeneous group $G'$, and take $m\in L^\infty(\beta_{\Lc'_{A'}})$ such that $\Kc_{\Lc'_{A'}}(m)\in L^1(G)$. 
	Then,\footnote{If $f\in L^1(G)$ and $\pi$ is an irreducible unitary representation of $G$, then $\Fc(f)(\pi)\coloneqq\int_G f(x) \pi(x^{-1})\,\dd x=\pi^*(f)$.  }
	\[
	\Fc(\Kc_{\Lc'_{A'}}(m))(\pi)=m(\dd \pi(\Lc'_{A'}))
	\]
	for almost every $[\pi]$ in the dual of $G$.
\end{proposition}

Before we state the next result, where we find relatively explicit formulae for the Plancherel measure and the integral kernel associated with $\Lc_A$, let us briefly comment on our techniques.
Thanks to the form of the Plancherel formula for $G$ (see~\cite{AstengoCowlingDiBlasioSundari}), we may basically reduce to study $\dd \pi_\omega(\Lc_A)$ for $\omega \neq 0$, or only for $\omega\in \Omega$. 
Therefore, the analysis of $\Lc_A$ is basically reduced to the case in which $n_2=1$. If $G$ is actually a Heisenberg group, then the Plancherel formula only involves the Bargmann-Fock representations $\pi_\lambda$ ($\lambda\neq 0$), and it is well-known that $\dd\pi_\lambda(\Lc_A)$ has an orthonormal basis of eigenfunctions such that the corresponding functions of positive type on $G$ are suitable Laguerre functions (cf.~\cite{HulanickiRicci}).
When $G$ has higher-dimensional centre (as in our case), then it splits into the product of a Heisenberg group and an abelian group, and the results are somewhat similar, even though the abelian factor causes some `superpositions' of different `layers' of the Plancherel measure associated with $\Lc_A$.

\begin{proposition}\label{prop:11:2}
	For every $\phi\in C_c(E_{\Lc_A})$,
	\[
	\begin{split}
	\int_{E_{\Lc_A}} \phi\,\dd \beta_{\Lc_A}&= \frac{\pi^{\frac{d}{2}}}{(2\pi)^{n_1+n_2+d} \Gamma\left( \frac{d}{2}\right) } \sum_{\gamma\in \N^{\overline h}} \binom{\vect{n_{1,\omega}}+\gamma-\vect{1}_{\overline h} }{\gamma} \times\\
	&\qquad\times\int_{ \R_+\times \gf_2^* } \phi(\mi_{\omega}(\vect{n_{1,\omega}}+2 \gamma)+\lambda, \omega(\vect{T})) \abs{\lambda}^{\frac{d}{2}-1}  \abs{\Pfaff(\omega)} \,\dd (\lambda,\omega).
	\end{split}
	\]
\end{proposition}

\begin{proof}
	We follow the construction of the Plancherel measure of~\cite{AstengoCowlingDiBlasioSundari} as in~\cite[4.4.1]{Martini}. 
	Take $\omega\in \gf_2^*\setminus W$ and $\tau\in P_{0,\omega}(\gf_1)$. 
	Let $\pi_{\omega,\tau}$ be an irreducible unitary representation of $G$ in a hilbertian space $H_{\omega,\tau}$ such that $\pi_{\omega,\tau}(x,t)=e^{i \omega(t)+ i\tau(x) } I_{H_{\omega,\tau}}$ for every $(x,t)\in P_{0,\omega}(\gf_1)\times \gf_2$.\footnote{Recall that such a  representation is uniquely determined up to unitary equivalence; cf.~\cite{AstengoCowlingDiBlasioSundari} and the references therein.}
	Then, for every $f\in L^2(G)$,
	\[
	\norm{f}_2^2= \frac{1}{(2\pi)^{n_1+n_2+d}} \int_{\gf_2^*} \int_{P_{0,\omega}(\gf_1)  } \norm{ \pi_{\omega,\tau}(f) }_2^2 \abs{\Pfaff(\omega)}\,\dd \tau\,\dd \omega.
	\]
	
	Now, it is well-known that there is a commutative family $(P_{\omega,\tau,\gamma})_{\gamma\in \N^{\overline h}}$ of self-adjoint projectors of $H_{\omega,\tau}$ such that $I_{H_{\omega, \tau}}=\sum_{\gamma\in \N^{\overline h}} P_{\omega,\tau,\gamma}$ pointwise, and such that for every $\gamma\in \N^{\overline h}$ we have $\tr P_{\omega, \tau,\gamma}=  \binom{\vect{n_{1,\omega}}+\gamma-\vect{1}_{\overline h}  }{\gamma}$ and (cf.~Proposition~\ref{prop:10:8})
	\[
	\dd \pi_{\omega, \tau}(\Lc_A) \cdot P_{\omega, \tau,\gamma}=(\abs{\tau}^2+\mi_\omega (\vect{n_{1,\omega}}+2 \gamma),\omega(\vect{T})) P_{\omega, \tau,\gamma}.
	\]
	Therefore, for every $\phi \in C^\infty_c(E_{\Lc_A})$, $\norm{\Kc_{\Lc_A}(\phi)}_2^2$ equals
	\[
	\begin{split}
	&\frac{1}{(2\pi)^{n_1+n_2+d}} \int_{\gf_2^*}\int_{P_{0,\omega}(\gf_1)} \sum_{\gamma\in \N^{\overline h}} \binom{\vect{n_{1,\omega}}+\gamma-\vect{1}_{\overline h}  }{\gamma} \abs*{\phi\left(\abs{\tau}^2+\mi_\omega (\vect{n_{1,\omega}}+2 \gamma),\omega(\vect{T})\right) }^2 \abs{\Pfaff(\omega)}\,\dd \tau\,\dd \omega,
	\end{split}
	\]
	whence stated formulae for $\beta_{\Lc_A}$.
\end{proof}

Now, let us make some remarks on how one may find an expression for $\chi_{\Lc_A}$; since that expression is not particularly illuminating, we shall omit to present it explicitly. We begin with a definition.

\begin{definition}
	Define
	\[
	\Phi_d\colon \R_+\ni x \mapsto \Gamma\left( \frac{d}{2} \right) \frac{ J_{\frac{d}{2}-1}(x)}{ \left( \frac{x}{2} \right)^{\frac{d}{2}-1} }=\Gamma\left( \frac{d}{2} \right)\sum_{k\in \N} \frac{(-1)^k x^{2 k} }{ 4^{ k} k!  \Gamma\left( k+\frac{d}{2} \right) }  ,
	\]
	where $J_{\frac{d}{2}-1}$ is the Bessel function (of the first kind) of order $\frac{d}{2}-1$.
\end{definition}

Observe first that from the Plancherel formula for $G$ which we stated in the proof of Proposition~\ref{prop:11:2} we deduce the following inversion formula:
\[
f(x,t)= \frac{1}{(2\pi)^{n_1+n_2+d}} \int_{\gf_2^*}\int_{P_{0,\omega}(\gf_1)} \tr( \pi_{\omega,\tau}(x,t)^* \pi_{\omega,\tau}(f) ) \abs{\Pfaff(\omega)}\,\dd \tau\,\dd\omega
\]
for every $f\in \Sc(G)$. If $\phi \in C^\infty_c(E_{\Lc_A})$, then $\Kc_{\Lc_A}(\phi)(x,t)$ equals, for almost every $(x,t)\in G$,
\[
\begin{split}
&\frac{1}{(2\pi)^{n_1+n_2+d}} \int_{\gf_2^*} \int_{P_{0,\omega}(\gf_1)} \sum_{\gamma\in \N^{\overline h}} \phi\left(\abs{\tau}^2+\mi_{\omega}(\vect{n_{1,\omega}}+2 \gamma),\omega(\vect{T})\right)\times\\
	&\qquad \qquad \qquad \qquad \qquad \qquad \qquad \qquad\times\tr (\pi_{\omega,\tau}(x,t)^* P_{\omega, \tau,\gamma} ) \abs{\Pfaff(\omega)}\,\dd \tau\,\dd \omega.
\end{split}
\]
Next,  if $\Lambda^m_\gamma(X)=\sum_{j=0}^\gamma \binom{\gamma+m}{\gamma-j} \frac{(-X)^j}{j!}$ denotes the $\gamma$-th Laguerre polynomial of order $m$, then
\[
\tr (\pi_{\omega,\tau}(x,t)^* P_{\omega, \tau,\gamma} )= e^{-\frac{1}{4}\abs{x_\omega}^2+  i \tau(x_{0,\omega}) +i \omega(t)}
\prod_{h=1}^{\overline{h}}  \Lambda^{n_{1,\omega,h}-1}_{\gamma_{h}}\left(\frac{1}{2} \abs{x_{\omega,h}}^2  \right) 
\]
by~\cite[Proposition 2]{HulanickiRicci} and~\cite[10.12 (41)]{Erdelyi}, while
\[
\mint{-}_{\partial B(0,1)\cap P_{0,\omega}(\gf_1)} e^{i \tau(x_{0,\omega})}\,\dd \Hc^{d-1}(\tau)= \Gamma\left( \frac{d}{2} \right)   \frac{ J_{\frac{d}{2}-1}\left( \abs{x_{0,\omega} } \right)  }{ \left( \frac{\abs{x_{0,\omega}}}{2}  \right)^{\frac{d}{2}-1  }}=\Phi_d(\abs{x_{0,\omega}}).
\]
One may then find formulae for $\chi_{\Lc_A}$.

\begin{remark}\label{oss:2}
	Let $T'_1,\dots, T'_n$ be $n$ homogeneous elements of the centre $\zf$ of $\gf$. 
	Let us show that the study of the family $(\Lc,-i T_1',\dots, - i T_n')$ can be reduced to that of the families of the form considered above on suitable $2$-step stratified groups.
	
	Notice that we may assume that there is $n'\in \Set{0,\dots, n}$ such that $T_j'\in \gf_2$ if and only if $j\meg n'$; let $\gf''$ be the vector subspace of $\gf$ generated by $T_{n'+1}',\dots, T_{n}'$, and observe that $\gf''\subseteq \gf_1$ by homogeneity. 
	Let $\gf'_1$ be the polar in $\gf_1$ of the $Q$-orthogonal complement of the polar of  $\gf''$ in $\gf_1^*$; define $\gf'\coloneqq \gf'_1\oplus \gf_2$. 
	Then,  $\gf$ is the direct sum of its ideals $\gf'$ and $\gf''$.
	Let $G'$ and $G''$ be the Lie subgroups of $G$ corresponding to $\gf'$ and $\gf''$, and let $\Lc'$ and $\Lc''$ be the sub-Laplacians on $G'$ and $G''$, respectively, corresponding to the restriction of $Q$ to $\gf_1'^*$ and $\gf''^*$. 
	By an abuse of notation, then, $\Lc=\Lc'+\Lc''$, so that the family $(\Lc, - i T_1',\dots, - i T_n')$ is equivalent to the family $(\Lc', - i T'_1,\dots, - i T'_n)$.
	Now, the family $(- i T'_{n'+1},\dots, - i T'_n)$ on $G''$ satisfies property $(RL)$ by classical Fourier analysis.
	Therefore, Theorem~\ref{teo:4} and its easy converse imply that the family $(\Lc, - i T_1',\dots, - i T_n')$ satisfies property $(RL)$  
	if and only if the family $(\Lc', - i T_1',\dots, - i T_{n'}')$ satisfies property $(RL)$. 
	Since this latter family is equivalent to a family of the form $(\Lc',-i T_1,\dots, - i T_{n_2'})$ for some $n_2'$ and for some choice of the basis $T_1,\dots, T_{n_2}$ of $\gf_2$, our assertion follows.
	Notice, however, that $G'$ may be an $MW^+$ group; we shall deal with $MW^+$ groups in a future paper.
	
	Similar arguments apply to property $(S)$ and the continuity of the integral kernel.
\end{remark}

\section{Property $(RL)$}

In this section we shall present several sufficient conditions for the validity of property $(RL)$. 
First of all, we observe that the spectrum of $\Lc_A$ is a semianalytic convex cone.
In addition,  we can basically ignore the Laguerre polynomials of higher order which appear in the Fourier inversion formula, thanks to Proposition~\ref{prop:2:9}. 
Indeed, with reference to the proof of Proposition~\ref{prop:11:2}, the `ground state', that is, the first eigenvalue of $\dd \pi_{\omega,\tau}(\Lc_A)$, is sufficient to cover the whole of $\sigma(\Lc_A)$, as $\omega$ and $\tau$ vary.
This fact leads to significant simplifications, as the basic Lemma~\ref{lem:9} shows.

We need to distinguish between the `full' family $\Lc_A$, for which we can prove continuity of the multipliers only on a dense subset of the spectrum \emph{in full generality} (cf.~Lemma~\ref{lem:9}), and the `partial' family $(\Lc, (-i T_1,\dots, - i T_{n_2'}))$ for $n_2'<n_2$, where by means of a deeper analysis we are able to prove property $(RL)$ in full generality (cf.~Theorem~\ref{prop:10}).
This latter result requires to deal with Radon measures defined on Polish spaces which are not necessarily locally compact.

Concerning the `full' family $\Lc_A$, as we observed above, we can prove in full generality that every integrable kernel corresponds to a multiplier which is continuous on a dense subset of the spectrum. Nevertheless, we can prove that property $(RL)$ holds in the following cases: when $P_0$ extends to a continuous function on $\gf_2^*\setminus\Set{0}$, for example when $W=\Set{0}$ or when $G$ is the product of an $MW^+$ group and a non-trivial abelian group (cf.~Theorem~\ref{prop:19:1}); when $G$ is a free $2$-step stratified group on an odd number of generators (cf.~Theorem~\ref{prop:19:5}).
In both cases, we make use of the simplified `inversion formula' for $\Kc_{\Lc_A}$ which is available in this case; in the second case, we employ the simple structure of free groups to prove that the $L^1$ kernels are invariant under sufficiently many linear transformations in order that the above-mentioned inversion formula give rise to a continuous multiplier.

\begin{lemma}\label{lem:9}
	Take $f\in L^1_{\Lc_A}(G)$. Then, $\Mc_{\Lc_A}(f)$ has a representative which is continuous on 
	\[
	\Set{ ( \mi_\omega(\vect{n_{1,\omega}} )  ,\omega(\vect{T}) )\colon \omega\in \gf_2^*  }\cup \Set{ ( \lambda ,\omega(\vect{T}) )\colon \omega\in \gf_2^*\setminus W ,\lambda \Meg \mi_\omega(\vect{n_{1,\omega}}) }.
	\]
\end{lemma}

\begin{proof}
	Fix a multiplier $m$ of $f$. 
	By Proposition~\ref{prop:2:9}, there is a negligible subset $N_1$ of $\gf_2^*$ such that for every $\omega\in \gf_2^*\setminus N_1$ there is negligible subset $N_{2,\omega}$ of $P_{0,\omega}(\gf_1)$ such that
	\[
	\pi_{\omega, \tau}^*(f)= m(\dd \pi_{\omega, \tau}(\Lc_A))
	\]
	for every $\tau \in P_{0,\omega}(\gf_1)\setminus N_{2,\omega}$. Notice that we may assume that $W\subseteq N_1$.
	Therefore, for every $\omega\in \gf_2^*\setminus N_1$ and for every $\tau\in P_{0,\omega}(\gf_1)\setminus N_{2,\omega}$,
	\[
	\begin{split}
	m(\mi_{\omega}(\vect{n_{1,\omega}})+\abs{\tau}^2, \omega(\vect{T}))&=\frac{1}{\tr P_{\omega,\tau,0}} \tr (m(\dd \pi_{\omega, \tau}(\Lc_A)) P_{\omega, \tau,0})\\
		&= \int_{G} f(x,t) e^{- \frac{1}{4}\abs{x_{\omega}}^2+ i \omega(t)+i \tau(x_{0,\omega})  }\,\dd (x,t). 
	\end{split}
	\]
	Now, for every $\omega\in \gf_2^*\setminus N_1$ there is negligible subset $N_{3,\omega}$  of $\R_+^*$ such that, for every $\lambda\in \R_+^*\setminus N_{3,\omega}$, we have $\Hc^{d-1}\big(\partial B_{P_{0,\omega}(\gf_1)}\big(0,\sqrt{\lambda} \big)\cap N_{2,\omega}\big)=0$.  
	Therefore, for every $\omega\in \gf_2^*\setminus N_1$ and for every $\lambda\in \R_+^*\setminus N_{3,\omega}$,
	\[
	\begin{split}
	m(\mi_{\omega}(\vect{n_{1,\omega}})+\lambda, \omega(\vect{T}))&=\mint{-}_{\partial B(0,\sqrt \lambda)} \int_G  f(x,t) e^{ -\frac{1}{4}\abs{x_{\omega}}^2+i \omega(t)+ i \tau(x_{0,\omega})   }\,\dd (x,t) \,\dd \Hc^{d-1}(\tau)\\
	&= \int_G  f(x,t) e^{ -\frac{1}{4}\abs{x_{\omega}}^2+i \omega(t) } \Phi_d\left(  \sqrt{\lambda} \abs{x_{0,\omega}} \right)\,\dd (x,t).
	\end{split}
	\]
	Now, the mapping
	\[
	(\omega, \lambda) \mapsto \int_G  f(x,t) e^{ -\frac{1}{4}\abs{x_{\omega}}^2+i \omega(t) } \Phi_d\left(  \sqrt{\lambda} \abs{x_{0,\omega}} \right)\,\dd (x,t)
	\]
	is continuous  on $[(\gf_2^*\setminus W)\times \R_+]\cup [\gf_2^*\times \Set{0}]$ by Proposition~\ref{prop:11:1}, so that by means of Tonelli's theorem we see that it induces a representative of $m$ which satisfies the conditions of the statement.
\end{proof}

\begin{theorem}\label{prop:19:1}
	Assume that $P_0$ can be extended to a continuous function on $\gf_2^*\setminus\Set{0}$. Then, $\Lc_A$ satisfies property $(RL)$.
\end{theorem}

Notice that, by polarization, $P_0$ has a continuous extension to $\gf_2^*\setminus\Set{0}$ if and only if $\abs{P_0(x)}$ has a continuous extension to $\gf_2^*\setminus\Set{0}$ for every $x\in \gf_1$.

In addition, observe that the hypotheses of the proposition hold in the following situations:
\begin{itemize}
	\item when $W=\Set{0}$, for example when $G$ is the free $2$-step nilpotent group on three generators;
	
	\item when $P_0$ is constant on  $\gf_2^*\setminus W$, for example when $G=G'\times\R^d$ for some $MW^+$ group $G'$, such as a product of Heisenberg groups.
\end{itemize}

\begin{proof}
	{\bf1.} Keep the notation of the proof of Lemma~\ref{lem:9}. 
	Assume first that $n_2=1$, so that $W=\Set{0}$. 
	In addition,  $\ker \dd_{\sigma_\omega}=\ker \dd_{\sigma_{-\omega}}$ for every $\omega\in \gf_2^*$, so that $P_0$ is constant on $\gf_2^*\setminus \Set{0}$. 
	The computations of the proof of Lemma~\ref{lem:9} then lead to the conclusion.
	
	{\bf2.} Denote by $\widetilde P_0$ the continuous extension of $P_0$ to $\gf_2^*\setminus \Set{0}$; observe that $\widetilde P_{0,\omega}$ is a self-adjoint projector of $\gf_1$ of rank $d$ for every non-zero $\omega\in \gf_2^*$.
	Take $f\in L^1_{\Lc_A}(G)$ and define, for every non-zero $\omega\in \gf_2^*$ and for every $\lambda \Meg 0$,
	\[
	m(\mi_{\omega}(\vect{n_{1,\omega}})+\lambda, \omega(\vect{T}))\coloneqq \int_G  f(x,t) e^{ -\frac{1}{4}\abs{x_{\omega}}^2+i \omega(t) } \Phi_d\left(  \sqrt{\lambda} \abs{\widetilde P_{0,\omega}(x)} \right)\,\dd (x,t),
	\]
	so that $f=\Kc_{\Lc_A}(m)$. 
	Then, $m$ is clearly continuous con $\sigma(\Lc_A)\setminus (\R\times\Set{0}^{n_2})$, and $m(\mi_{r \omega}(\vect{n_{1, \omega}})+\lambda, r \omega(\vect{T}))$ converges to 
	\[
	\int_G  f(x,t) \Phi_d\left(  \sqrt{\lambda} \abs{\widetilde P_{0,\omega}(x)} \right)\,\dd (x,t)
	\]
	as $r\to 0^+$, uniformly as $\omega$ runs through the unit sphere $S$ of $\gf_2^*$. 
	Therefore, it will suffice to prove that the above integrals do not depend on $\omega\in S$ for every $\lambda\Meg 0$.
	Indeed, Proposition~\ref{prop:3} implies that, for every $\omega\in S$,
	\[
	(\pi_\omega)_*(f)=\Kc_{\dd\pi_\omega(\Lc_A)}(m).
	\]
	Now,~{\bf1} above implies that the family $\dd \pi_\omega(\Lc_A)$ satisfies property $(RL)$. 
	Then, Proposition~\ref{prop:3}  implies that
	\[
	(\pi_0)_*(f)\in L^1_{\dd\pi_0(\Lc_A)}(G_0);
	\] 
	in addition, $\dd \pi_0(\Lc_A)$ is identified with $(\Delta, 0,\dots,0)$, where $\Delta$ is the (positive) Laplacian associated with the scalar product $\widehat Q$ on $\gf_1$. 
	Then,
	\[
	\int_G f(x,t)  \Phi_d\left(  \sqrt{\lambda} \abs{\widetilde P_{0,\omega}(x)} \right)\,\dd (x,t)
	=\int_{\gf_1}(\pi_0)_*(f)(x)  \Phi_d\left(  \sqrt{\lambda} \abs{\widetilde P_{0,\omega}(x)} \right)\,\dd x,
	\]
	whence the assertion since $(\pi_0)_*(f)$ is rotationally invariant.	
\end{proof}

\begin{remark}\label{oss:1}
	Let $G$ be $\Hd^1\times \R$, where $\Hd^1$ is the $3$-dimensional Heisenberg group. If $\Lc$ is the standard sub-Laplacian on $\Hd^1$, $T$ is a basis of the centre of the Lie algebra of $\Hd^1$, and $\Delta$ is the (positive) Laplacian on $\R$, then $(\Lc+\Delta, i T)$ satisfies property $(RL)$ by Theorem~\ref{prop:19:1}, but it is easily seen that its integral kernel does not admit any continuous representatives.
\end{remark}

When $G$ is a free group, we can remove the assumption that $P_0$ has a continuous extension.

\begin{theorem}\label{prop:19:5}
	Assume that $G$ is a free $2$-step stratified group on an odd number of generators. 
	Then, $\Lc_A$ satisfies property $(RL)$.
\end{theorem}

\begin{proof}
	Take $f\in L^1_{\Lc_A}(G)$; by Lemma~\ref{lem:9}, $f$ has a multiplier $m$  which is continuous on $\sigma(\Lc_A)\setminus (\R\times W)$.
	Now,
	\[
	(\pi_\omega)_*(f)(x,t) =\int_{\omega(t')= t} f(x, t')\,\dd t'
	\]
	for almost every $(x,t)\in G_\omega$. 
	Then, Proposition~\ref{prop:3} implies that $(\pi_\omega)_*(f)$ is invariant under the isometries which restrict to the identity on $(\ker \dd_{\sigma_\omega})^\perp$, for every $\omega \in \gf_2^*\setminus W$; indeed, $\dd \pi_\omega(\Lc_A)$ is invariant under such isometries, and these isometries  are group automorphisms.
	Next, take $\omega\in W$ and an isometry $U$ of $G_\omega$ which restricts to the identity on $(\ker \dd_{\sigma_\omega})^\perp$. 
	Since $\dim \ker \dd_{\sigma_\omega}$ is odd, there must be some $v\in \ker \dd_{\sigma_\omega}$ such that $U\cdot v=\pm v$. 
	Let $V$ be the orthogonal complement of $\R v$ in $\ker \dd_{\sigma_\omega}$, so that $V$ is $U$-invariant.
	Now, let $\sigma_V$ be a standard symplectic form on the hilbertian space $V$,\footnote{That is, choose a symplectic form $\sigma_V$ on $V$ so that $V$ admits an orthonormal basis (relative to the scalar product) which is also a symplectic basis (relative to $\sigma_V$).} and define $\omega_p$, for every $p\in \N$, so that
	\[
	\sigma_{\omega_p}=\sigma_\omega+ 2^{-p} \sigma_V;
	\]
	this is possible since $G$ is a free $2$-step stratified group.
	Then, $\omega_p$ belongs to $\gf_2^*\setminus W$ and converges to $\omega$. 
	In addition, $(\pi_{\omega_p})_*(f)$	is $U$-invariant thanks to Proposition~\ref{prop:3}.
	Now, it is easily seen that $(\pi_{\omega_p})_*(f)$ converges to $(\pi_\omega)_*(f)$ in $L^1(\gf_1 \oplus \R)$, so that $(\pi_\omega)_*(f)$ is $U$-invariant.
	
	Then, the mapping
	\[
	m_1\colon \R_+\times(\gf_2^*\setminus W) \ni(\lambda,\omega) \mapsto \int_G  f(x,t) e^{ -\frac{1}{4}\abs{x_{\omega}}^2+i \omega(t) } \Phi_1\left(  \sqrt{\lambda} \abs{x_{0,\omega}} \right)\,\dd (x,t),
	\]
	extends to a continuous function on $\gf_2^*\times \R_+$. 
	Now, clearly $m(\lambda, \omega(\vect{T}))=m_1(\lambda- \mi_\omega(\vect{n_{1,\omega}}),\omega)$ for every $(\lambda, \omega(\vect{T}))\in \sigma(\Lc_A)$; the assertion follows.		
\end{proof}

\begin{theorem}\label{prop:10}
	Take $n_2'<n_2$. 
	Then, the family $(\Lc, (-i T_j)_{j=1,\dots,n_2'})$ satisfies property $(RL)$.
\end{theorem}

\begin{proof}
	Define $\Lc'_{A'}=(\Lc, (-i T_j)_{j=1,\dots, n_2'})$, and let $L\colon E_{\Lc_A}\to E_{\Lc'_{A'}}$ be the unique linear mapping such that $\Lc'_{A'}=L(\Lc_A)$. 
	Until the end of the proof, we shall identify $\gf_2^*$ with $\R^{n_2}$ by means of the mapping $\omega \mapsto \omega (\vect{T})$.
	In addition, define $X\coloneqq (\sigma(\Lc_A) \setminus  W) \cup \partial \sigma(\Lc_A)$, so that $X$ is a Polish space by~\cite[Theorem 1 of Chapter IX, § 6, No.\ 1]{BourbakiGT2}.
	Let $\beta$ be the (Radon) measure induced by $\beta_{\Lc_A}$ on $X$, so that $\Supp{\beta}=X$. 
	Let $L'$ be the restriction of $L$ to $X$. Since  $\sigma(\Lc_A)$ is a convex cone by Corollary~\ref{cor:11:2} and since $ W$ is $\beta_{\Lc_A}$-negligible, Proposition~\ref{prop:A:8} implies that $\beta$ is $L'$-connected. 
	
	Now, Proposition~\ref{prop:A:6} implies that $\beta$ has a disintegration $(\beta_{\lambda'})_{\lambda'\in E_{\Lc'_{A'}}}$ such that $\beta_{\lambda'}$ is equivalent to $\chi_{L'^{-1}(\lambda')}\cdot \Hc^{n_2-n_2'}$ for $\beta_{\Lc'_{A'}}$-almost every $\lambda'\in E_{\Lc'_{A'}}$. 
	Observe that $L^{-1}(\lambda')\cap \sigma(\Lc_A)$ is a convex set of dimension $n_2-n_2'$ for $\beta_{\Lc'_{A'}}$-almost every $\lambda'\in E_{\Lc'_{A'}}$. 
	In addition, $ W\cap L^{-1}(\lambda')$ is an algebraic variety of dimension at most $n_2-n_2'-1$ for $\beta_{\Lc'_{A'}}$-almost every $\lambda'\in E_{\Lc'_{A'}}$, for otherwise $\Hc^{n_2+1}(W)$ would be non-zero, which is absurd.
	Therefore, $\Supp{\beta_{\lambda'}}=L'^{-1}(\lambda')$ for $\beta_{\Lc'_{A'}}$-almost every $\lambda'\in E_{\Lc'_{A'}}$.  
	
	Now, take $m_0\in L^\infty(\beta_{\Lc_A})$ so that $\Kc_{\Lc'_{A'}}(m_0)\in L^1(G)$. 
	Let us prove that $m_0$ has a continuous representative.
	Indeed, Lemma~\ref{lem:9} implies that there is a continuous function $m_1$ on $X$ such that $m_0\circ L'=m_1$ $\beta$-almost everywhere.
	Hence, Proposition~\ref{prop:A:7} implies that there is a function $m_2\colon \sigma(\Lc'_{A'})\to \C$ such that $m_2\circ L'=m_1$. Since the mapping $L\colon \partial \sigma(\Lc_A)\to \sigma(\Lc'_{A'})$ is proper and onto, and since $\partial \sigma(\Lc_A)\subseteq X$, it follows that $m_2$ is continuous.
	The assertion follows (cf.~\cite[Corollary to Theorem 2 of Chapter IX, § 4, No.\ 2]{BourbakiGT2}). 
\end{proof}

\section{Property $(S)$}

The results of this section are basically a generalization of the techniques employed in~\cite{AstengoDiBlasioRicci,AstengoDiBlasioRicci2}. 

Theorem~\ref{prop:20:2} applies, for example, to the free $2$-step nilpotent group on three generators.
Notice that we need to impose the condition $W=\Set{0}$ since our methods cannot be used to infer any kind of regularity on $W\setminus \Set{0}$; for example, in general our auxiliary functions $\abs{x_\omega}^2$ and $P_0$ are not differentiable on $W$.
Nevertheless, this does not mean that property $(S)$ cannot hold when $W\neq \Set{0}$; as a matter of fact, Theorem~\ref{prop:20:3} shows that this happens for a product of free $2$-step stratified groups on $3$ generators and a suitable sub-Laplacian thereon.

In order to simplify the notation, we define $\Sc(G,\Lc_A)\coloneqq\Kc_{\Lc_A}(\Sc(E_{\Lc_A}))$.

\medskip

We begin with a lemma which will allow us to get some `Taylor expansions' of multipliers corresponding to Schwartz kernels under suitable hypotheses. Its proof is modelled on a technique due to D.\ Geller~\cite[Theorem 4.4]{Geller}.
We state it in a slightly more general context.

\begin{lemma}\label{lem:20:3}\label{lem:20:4}
	Let $\Lc_A$ be a Rockland family on a homogeneous group $G'$, and let $T'_1,\dots, T'_n$ be a free family of elements of the centre of the Lie algebra $\gf'$ of $G'$. 
	Let $\pi_1$ be the canonical projection of $G'$ onto its quotient by the normal subgroup $\exp(\R T'_1)$, and assume that the following hold:
	\begin{itemize}
		\item $(\Lc_A, i T'_1,\dots, i T'_n )$ satisfies property $(RL)$;
		
		\item  $\dd \pi_1(\Lc_A, i T'_{2},\dots, i T'_n)$ satisfies property $(S)$.
	\end{itemize}
	
	Take $\phi\in \Sc_{(\Lc_A, i T'_1,\dots, i T'_n )}(G')$. 
	Then, there are two families $(\widetilde \phi_\gamma)_{\gamma\in \N^{n}}$ and $(\phi_\gamma)_{\gamma\in \N^{n}}$  of elements of $\Sc(G',\Lc_A)$ and $\Sc_{(\Lc_A, i T'_1,\dots, i T'_n )}(G')$, respectively, such that 
	\[
	\phi=\sum_{\abs{\gamma}<h} \vect{T}'^\gamma \widetilde \phi_\gamma+\sum_{\abs{\gamma}=h} \vect{T}'^\gamma  \phi_\gamma
	\]
	for every $h\in \N$.
\end{lemma}

\begin{proof}
	For every $k\in \Set{1,\dots, n}$, let $G'_k$ be the quotient of $G'$ by the normal subgroup $\exp(\R T'_k)$. 
	Endow $\gf'$ with a scalar product which turns $(T'_1,\dots, T'_n)$ into an orthonormal family.
	Then, Proposition~\ref{prop:3}  implies that $(\pi_1)_*(\phi)\in \Sc_{\dd \pi_1(\Lc_A, i T'_{2},\dots, i T'_{n})}(G'_1)$, so that there is $\widetilde m_1\in \Sc(E_{\dd \pi_1(\Lc_A, i T'_{2},\dots, i T'_{n})})$ such that $(\pi_1)_*(\phi)=\Kc_{\dd \pi_1(\Lc_A, i T'_{2},\dots, i T'_{n})}(\widetilde m_1)$.
	Therefore, if we define $\widetilde \phi_{0,1}\coloneqq \Kc_{(\Lc_A, i T'_{2},\dots, i T'_{n})}(\widetilde m_1)$, then Proposition~\ref{prop:3} implies that $(\pi_1)_*(\phi-\widetilde \phi_{0,1})=0$.
	In other words,
	\[
	\int_{\R} (\phi-\widetilde \phi_{0,1})(\exp(x+s T'_1))\,\dd s=0
	\]
	for every $x\in  T_1'^\perp$. 
	Identifying $\Sc(G') $ with $\Sc(\R T'_1;\Sc(T_1'^\perp))$, by means of a simple consequence of the classical Hadamard's lemma we see that there is $\phi_1 \in \Sc(G')$ such that
	\[
	\phi=\widetilde \phi_{0,1}+ T'_1 \phi_1.
	\]
	Now, let us prove that $\phi_1\in \Sc_{(\Lc_A, i T'_{1},\dots, i T'_{n})}(G')$. 
	Indeed,  
	\[
	T'_1 \Kc_{(\Lc_A, i T'_{2},\dots, i T'_{n})}\Mc_{(\Lc_A, i T'_{2},\dots, i T'_{n})}(\phi_1)=\phi-\widetilde \phi_{0,1}=T'_1 \phi_1.
	\] 
	Since clearly $\Kc_{(\Lc_A, i T'_{2},\dots, i T'_{n})}\Mc_{(\Lc_A, i T'_{2},\dots, i T'_{n})}(\phi_1)\in L^2(G')$, and since $T'_1$ is one-to-one on $L^2(G')$, the assertion follows.
	If $n\Meg 2$, then we can apply the same argument to $\widetilde \phi_{0,1}$ considering the quotient $G_2'$, since we already know that $\widetilde \phi_{0,1}$ has a Schwartz multiplier. Then, we obtain $\widetilde \phi_{0,2}\in \Sc(G',(\Lc_A, i T'_{3},\dots, i T'_{n}))$ and $\phi_2\in \Sc_{(\Lc_A, i T'_{1},\dots, i T'_{n})} (G')$ such that
	\[
	\phi=\widetilde \phi_{0,2}+ T'_1 \phi_1+ T'_2 \phi_2.
	\]
	Iterating this procedure, we eventually find functions $\widetilde \phi_0\in \Sc(G',\Lc_A)$ and $\phi_1,\dots, \phi_{n}\in \Sc_{(\Lc_A, i T'_{1},\dots, i T'_{n})}(G')$ such that
	\[
	\phi=\widetilde \phi_0+\sum_{k=1}^{n} T'_k \phi_k.
	\]
	The assertion follows proceeding inductively.
\end{proof}

Notice that, if $G$ is abelian and $\Lc$ is a Laplacian on $G$, then $\Lc$ satisfies properties $(RL)$ and $(S)$ (cf.~\cite{Whitney2}).

\begin{theorem}\label{prop:20:2}
	Assume that $W=\Set{0}$.
	Then, $(\Lc, (-i T_k)_{k=1}^{n_2'})$ satisfies property $(S)$ for every $n_2'\meg n_2$. 
\end{theorem}

\begin{proof}
	Notice that Theorems~\ref{prop:19:1} and~\ref{prop:10} imply that $(\Lc, (-i T_k)_{k=1}^{n_2'})$ satisfies property $(RL)$. 
	Therefore, by means of Corollary~\ref{cor:A:7}  we see that it will suffice to prove the assertion for $n_2'=n_2$. In addition, the abelian case, that is, the case $n_2=0$ has already been considered.
	We proceed by induction on $n_2\Meg 1$. 
	
	{\bf1.} Observe first that the abelian case, Theorem~\ref{prop:19:1}, and Lemma~\ref{lem:20:3} imply that we may find a family $(\widetilde \phi_\gamma)$ of elements of $\Sc(G,\Lc)$, and a family $(\phi_\gamma)$ of elements of $\Sc_{\Lc_A}(G)$ such that
	\[
	\phi=\sum_{\abs{\gamma}< h} \vect{T}^\gamma\widetilde \phi_\gamma+ \sum_{\abs{\gamma}=h}\vect{T}^\gamma \phi_\gamma
	\] 
	for every $h\in \N$.
	
	Define $\widetilde m_\gamma\coloneqq \Mc_{\Lc}(\widetilde \phi_\gamma)\in \Sc(\sigma(\Lc))$ and $m_\gamma\coloneqq \Mc_{\Lc_A}(\phi_\gamma)\in C_0(\beta_{\Lc_A})$ for every $\gamma$.
	Then, 
	\[
	m_0(\lambda, \omega) =\sum_{\abs{\gamma}<h} \omega^\gamma \widetilde m_\gamma(\lambda)+ \sum_{\abs{\gamma}=h} \omega^\gamma m_\gamma(\lambda, \omega)
	\]
	for every $h\in \N$ and for every $(\lambda, \omega)\in \sigma(\Lc_A)$.

	By a vector-valued version of Borel's lemma (cf.~\cite[Theorem 1.2.6]{Hormander2} for the scalar, one-dimensional case), we see that there is $\widehat m\in C^\infty_c(\R^{n_2}; \Sc(\R) )$ such that $\widehat m^{(\gamma)}(0)=\widetilde  m_\gamma$ for every $\gamma\in \N^{n_2}$. 
	Interpret $\widehat m$ as an element of $\Sc(E_{\Lc_A})$.
	Reasoning on $m-\widehat m$, we may reduce to the case in which $\widetilde m_\gamma=0$ for every $\gamma$; then, we shall simply write $m$ instead of $m_0$.
	
	{\bf2.}	Consider the norm $N\coloneqq \mi(\vect{n_{1}})$ on $\gf_2^*$ and let $S$ be the associated unit sphere. 
	Define $\sigma(\omega)\coloneqq \frac{\omega}{N(\omega)}$ for every $\omega\in \gf_2^*\setminus \Set{0}$.
	Then, 
	the mapping 
	\[
	S\ni \omega \mapsto (\pi_\omega)_*(\phi)\in \Sc(\gf_1 \oplus \R)
	\]
	is of class $C^\infty$. 
	Fix $\omega_0\in S$. 
	It is not hard to  see that we may find a dilation-invariant open neighbourhood $U$ of $\omega_0$ and an analytic mapping $\psi\colon U\times (\gf_1\oplus \R)\to \R^{2 n_1}\times \R\times \R^d$ such that, for every $\omega\in U$, $\psi_\omega\coloneqq \psi(\omega,\,\cdot\,)$ is an isometry of $\gf_1 \oplus \R$ onto $\R^{2 n_1}\times \R\times \R^d$ such that $\psi_\omega(P_{0,\omega}(\gf_1))=\Set{0}\times \R^d$ and $\psi_\omega(\Set{0}\times \R)=\Set{0}\times \R\times \Set{0}$.
	Take $\omega\in U$. 
	By transport of structure, we may put on $\R^{2 n_1}\times \R$ a group structure for which $\R^{2 n_1}\times \R$ is isomorphic to $\Hd^{n_1}$ and which turns $\psi_\omega$ into an isomorphism of Lie groups, where, $\gf_1\oplus \R$ is endowed with the structure induced by its identification with $G_\omega$.\footnote{Obviously, this structure depends on $\omega$.} 
	Then, there is a sub-Laplacian $\Lc'_\omega$ on $\R^{2 n_1}\times \R$ such that, if $T$ denotes the derivative along  $\Set{0}\times \R\subseteq\R^{2 n_1}\times \R$ and $\Delta$ is the standard (positive) Laplacian on $\R^d$, then
	\[
	\dd (\psi_\omega\circ\pi_\omega)(\Lc_A)=(\Lc'_\omega+\Delta, \omega(\vect{T}) T).
	\]  
	Then, Proposition~\ref{prop:3} and Lemma~\ref{lem:3}  imply that 
	\[
	(\psi_\omega\circ\pi_\omega)_*(\phi_\gamma)((y,t),\,\cdot\,)\in \Sc_\Delta(\R^d)
	\]
	for every $(y,t)\in \R^{2 n_1}\times \R$ and for every $\gamma\in \N^{n_2}$. 
	Define
	\[
	\widehat \phi_\gamma\colon (U\cap S) \times \R_+\times( \R^{2 n_1}\times\R)\ni (\omega, \xi, (y,t))\mapsto \Mc_\Delta ((\psi_\omega\circ\pi_\omega)_*(\phi_\gamma)((y,t),\,\cdot\,)  )(\xi),
	\]
	so that $\widehat \phi_\gamma(\omega,\,\cdot\,,(y,t))\in \Sc(\R_+)$ for every $\omega\in U\cap S$ and for every $(y,t)\in \R^{2 n_1}\times \R$ since $\Delta $ satisfies property $(S)$. 
	In addition,  the mapping 
	\[
	\omega \mapsto [ (y,t)\mapsto (\psi_\omega\circ\pi_\omega)_*(\phi_\gamma)((y,t),\,\cdot\,)  ]
	\]
	belongs to $\Ec(S\cap U; \Sc(\R^{2 n_1}\times \R; \Sc_\Delta(\R^d) ))$, 
	so that the mapping
	\[
	\omega \mapsto [ (y,t)\mapsto \widehat\phi_\gamma(\omega,\,\cdot\,,(y,t))  ]
	\] 
	belongs to $\Ec(S\cap U; \Sc(\R^{2 n_1}\times \R; \Sc_\R(\R_+) ))$.
	Now, observe that the mapping 
	\[
	U\ni \omega \mapsto \psi_\omega^{-1}\in \Lc(\R^{2 n_1}\times \R\times \R^d; \gf_1 \oplus \R)
	\]
	is of class $C^\infty$, so that also the mapping
	\[
	f\colon U\times \R^{n_1}\ni(\omega,y)\mapsto \abs{ ( \psi_{\sigma(\omega)}^{-1}(y,0,0)  )_{\omega} }^2
	\]
	is of class $C^\infty$, thanks to Proposition~\ref{prop:11:1}. 
	In addition, by means of Proposition~\ref{prop:11:2} we see that
	\[
	\begin{split}
	m_\gamma(\xi+N(\omega), \omega(\vect{T}))= \int_{\R^{2 n_1}\times \R} \widehat \phi_\gamma(\sigma(\omega), \xi, (y,t)) e^{ -\frac{1}{4} f(\omega, y)+i N(\omega)t }\,\dd (y,t)
	\end{split}
	\]
	for every $\gamma\in \N^{n_2}$, for every $\omega\in U$ and for every $\xi\Meg 0$.
	Therefore, the preceding arguments and some integrations by parts show that
	\[
	\begin{split}
	m(\xi+N(\omega),\omega(\vect{T}))
	&=\sum_{\abs{\gamma}=h}\sigma(\omega(\vect{T}))^\gamma\int_{\Hd^{n_1}} T^h\widehat \phi_\gamma(\sigma(\omega),\xi,(y,t)) e^{-\frac{1}{4} f(\omega,y)+i N(\omega) t }\,\dd (y,t)\\
	&=\sum_{\abs{\gamma}=h}(-i\omega(\vect{T}))^\gamma\int_{\Hd^{n_1}} \widehat \phi_\gamma(\sigma(\omega),\xi,(y,t)) e^{-\frac{1}{4} f(\omega,y)+i N(\omega) t }\,\dd (y,t)
	\end{split}
	\]
	for every $h\in \N$, for every $\omega\in U$ and for every $\xi \Meg 0$. 
	Now, fix $p_1,p_2,p_3\in \N$, and take $h\in\N$.
	Apply Faà di Bruno's formula and integrate by parts $p_3$ times in the $t$ variable.
	Then,  there is a constant $C>0$ such that
	\[
	\begin{split}
	\abs{ 	(\partial_1^{p_1} \partial_2^{p_2}m )(\xi,\omega(\vect{T})) }&\meg C N(\omega)^{h-p_2-p_3}(1+N(\omega))^{p_2} \int_{\Hd^{n_1}} (1+\abs{(y,t)})^{2 p_2}\times\\
	& \qquad\times \max_{\substack{\abs{\gamma}=h\\q_2+q_3=0,\dots, p_2}} \abs{  \widehat \phi_\gamma^{(p_1+p_3+q_2+q_3)}(\sigma(\omega), \xi-N(\omega), (y,t)) } \,\dd (y,t)
	\end{split}
	\]
	for every $(\xi,\omega(\vect{T}))\in \open{\sigma(\Lc_A)}\cap (\R\times U)$.
	Here, $\abs{(y,t)}=\abs{y}+\sqrt{\abs{t}}$ is a homogeneous norm on $\R^{2 n_1}\times \R$.
	
	Now, take a compact subset $K$ of $U\cap S$. 
	Then, the properties of the $\widehat \phi_\gamma$ imply that for every $p_4\in \N$ there is a constant $C'$ such that
	\[
	\abs{ \widehat \phi_\gamma^{(q)}(\omega, \xi, (y,t)) }\meg \frac{C'}{ (1+\xi)^{p_4}(1+\abs{(y,t)})^{2 p_2+ 2 n_1+3}  }
	\]
	for every $\gamma$ with length $h$, for every $q=0,\dots, p_1+p_2+p_3 $, for every $\omega\in K$, for every $\xi\Meg 0$ and for every $(y,t)\in \R^{2 n_1}\times \R$.
	Therefore, there is a constant $C''>0$ such that
	\[
	\abs{ 	(\partial_1^{p_1} \partial_2^{p_2}m )(\xi,\omega(\vect{T})) }\meg C'' N(\omega)^{h-p_2-p_3}\frac{(1+N(\omega))^{p_2}}{ (1+ \xi-N(\omega)   )^{p_4}   } 
	\]
	for every $(\xi,\omega(\vect{T}))\in \open{\sigma(\Lc_A)}\cap (\R\times U)$ such that $\sigma(\omega)\in K$. 
	By the arbitrariness of $U$ and $K$, and by the compactness of $S$, we see that we may take $C''$ so that the preceding estimate holds for every $(\xi,\omega(\vect{T}))\in \open{\sigma(\Lc_A)}\cap (\R \times (\R^{n_2}\setminus\Set{0}))$.
	
	Now, taking $h-p_3>p_2$ we see that $\partial_1^{p_1} \partial_2^{p_2}m$  extends to a continuous function on $\sigma(\Lc_A)$ which vanishes on $\R_+\times \Set{0}$. 
	If $N(\omega)\meg \frac{1}{3}$, then take $h-p_3=p_2$ and observe that
	\[
	\frac{1}{3} + \xi+ N(\omega) \meg \frac{2}{3}+\xi\meg 1+\xi-N(\omega)
	\] 
	for every $\xi \Meg N(\omega)$.
	On the other hand, if $N(\omega)\Meg \frac{1}{3}$, then take $p_3=p_4+h$ and observe that
	\[
	1+\xi+N(\omega) \meg (1+2 N(\omega))(1+\xi-N(\omega))\meg 5 N(\omega) (1+\xi-N(\omega))
	\]
	for every $\xi \Meg N(\omega)$.
	Hence, for every $p_4\in \N$ we may find a constant $C'''>0$ such that
	\[
	\abs{ 	(\partial_1^{p_1} \partial_2^{p_2}m )(\xi,\omega(\vect{T})) }\meg C''' \frac{1}{ (1+ \xi+N(\omega)   )^{p_4}   } 
	\]
	for every $\xi\Meg N(\omega)$.
	Now, extending~\cite[Theorem 5 of Chapter VI]{Stein} to the case of Schwartz functions in the spirit of~\cite[Theorem 6.1]{AstengoDiBlasioRicci2}, we see that $m\in \Sc_{E_{\Lc_A}}(\sigma(\Lc_A))$.
\end{proof}

\begin{theorem}\label{prop:20:3}
	Assume that $G$ is the product of a finite family $(G_\eta)_{\eta\in H}$ of $2$-step stratified groups which do not satisfy the $MW^+$ condition; endow each $G_\eta$ with a sub-Laplacian $\Lc_\eta$ and assume that $(\Lc_\eta, i \Tc_\eta)$ satisfies property $(RL)$ (resp.\ $(S)$) for some finite family $\Tc_\eta$ of elements of the second layer of the Lie algebra of $G_\eta$.
	Define $\Lc\coloneqq \sum_{\eta\in H} \Lc_\eta$ (on $G$), and let $\Tc$ be a finite family of elements of the vector space generated by the $\Tc_\eta$.
	Then, the family $(\Lc, - i \Tc)$ satisfies property $(RL)$ (resp.\ $(S)$). 
\end{theorem}

\begin{proof}
	Observe first that, by means of Propositions~\ref{prop:A:7},~\ref{prop:A:8}, and~\ref{prop:A:6}, and Corollary~\ref{cor:A:7}, we may reduce to the case in which $\Tc$ is the union of the $\Tc_\eta$. 
	Then, Theorems~\ref{teo:4},~\ref{teo:5},~\ref{prop:19:1}, and~\ref{prop:20:2}, imply that the family $(\Lc_H, - i \Tc)$ satisfies property $(RL)$ (resp.\ $(S)$).
	Therefore, the assertion follows easily from Propositions~\ref{prop:A:7},~\ref{prop:A:8}, and~\ref{prop:A:6}, and Corollary~\ref{cor:A:7}.	
\end{proof}

\section*{Acknowledgements}
I would like to thank professor F.\ Ricci for patience and guidance, as well as for many inspiring discussions and for the numerous suggestions concerning the redaction of this manuscript.	
I would also like to thank Dr A.\ Martini and L.\ Tolomeo for some discussions concerning their work.

\end{document}